\documentclass{amsart}
\usepackage{amsmath,amsthm,amsfonts,amssymb,latexsym,mathrsfs,graphicx}
\usepackage{pgf,tikz}
\usepackage{ulem}
\usetikzlibrary{arrows}

\usepackage{hyperref}

\usepackage{enumerate}
\usepackage[shortlabels]{enumitem}
\usepackage{color}

\usepackage[latin1]{inputenc}
\usepackage{bm} 

\usepackage{pstricks-add}
\usepackage{graphicx}
\usepackage{pgf,tikz,pgfplots}
\usepackage{mathrsfs}
\usetikzlibrary{arrows}
\definecolor{ududff}{rgb}{0.30196078431372547,0.30196078431372547,1}

\newtheorem{theorem}{Theorem}[section]

\newtheorem{lemma}[theorem]{Lemma}

\newtheorem*{Thm1}{Theorem 1}
\newtheorem*{Thm2}{Theorem 2}

\theoremstyle{definition}

\newtheorem{rem}[theorem]{Remark}

\newenvironment{enumeratei}{\begin{enumerate}[\upshape (a)]}
    {\end{enumerate}}

\def\irr#1{{\rm Irr}(#1)}
\def\cent#1#2{{\bf C}_{#1}(#2)}

\def\syl#1#2{{\rm Syl}_#1(#2)}

\def\nor{\trianglelefteq\,}
\def\oh#1#2{{\bf O}_{#1}(#2)}
\def\zent#1{{\bf Z}(#1)}
\def\V#1{{V}(#1)}
\def\sbs{\subseteq}

\def\fit#1{{\bf F}(#1)}
\def\frat#1{{\bf \Phi}(#1)}

\newcommand{\N}{{\mathbb N}}

\newcommand{\F}{{\mathbb F}}
\newcommand{\K}{{\mathbb K}}

\def\irr#1{{\rm Irr}(#1)}

\def\cent#1#2{{\bf C}_{#1}(#2)}

\def\syl#1#2{{\rm Syl}_#1(#2)}
\def\nor{\trianglelefteq\,}
\def\norm#1#2{{\bf N}_{#1}(#2)}
\def\oh#1#2{{\bf O}_{#1}(#2)}

\def\zent#1{{\bf Z}(#1)}
\def\sbs{\subseteq}

\def\fit#1{{\bf F}(#1)}

\def\SL#1{{\rm SL}_{2}(#1)}
\def\PSL#1{{\rm PSL}_{2}(#1)}

\def\V#1{{\rm V}(#1)}

\def\irr#1{{\rm Irr}(#1)}

\def\cd#1{{\rm cd}(#1)}

\def\cent#1#2{{\bf C}_{#1}(#2)}
\def\syl#1#2{{\rm Syl}_#1(#2)}
\def\oh#1#2{{\bf O}_{#1}(#2)}

\def\zent#1{{\bf Z}(#1)}

\def\ker#1{{\rm ker}(#1)}
\def\norm#1#2{{\bf N}_{#1}(#2)}

\def\sbs{\subseteq}

\def \o#1{\overline{#1}}
\mathchardef\coso="2023

\def \nq{\mathcal{N}_q}

\begin{document}

\title[Degree graphs with a cut-vertex. III]{Non-solvable groups whose character degree graph has a cut-vertex. III}

\author[]{Silvio Dolfi}
\address{Silvio Dolfi, Dipartimento di Matematica e Informatica U. Dini,\newline
Universit\`a degli Studi di Firenze, viale Morgagni 67/a,
50134 Firenze, Italy.}
\email{silvio.dolfi@unifi.it}

\author[]{Emanuele Pacifici}
\address{Emanuele Pacifici, Dipartimento di Matematica e Informatica U. Dini,\newline
Universit\`a degli Studi di Firenze, viale Morgagni 67/a,
50134 Firenze, Italy.}
\email{emanuele.pacifici@unifi.it}

\author[]{Lucia Sanus}
\address{Lucia Sanus, Departament de Matem\`atiques, Facultat de Matem\`atiques, \newline Universitat de Valencia, Burjassot,
46100 Valencia, Spain.}
\email{lucia.sanus@uv.es}


\thanks{The first two authors are partially supported by the italian INdAM-GNSAGA. The research of the third author is partially supported by Ministerio de Ciencia e Innovaci\'on PID2019-103854GB-I00.  This research has been carried out during a visit of the third author at the Dipartimento di Matematica e Informatica ``Ulisse Dini" (DIMAI) of Universit\`a degli Studi  di Firenze.  She thanks  DIMAI for the financial support and hospitality.}


\begin{abstract} 
Let $G$ be a finite group. Denoting by $\cd{G}$ the set of the degrees of the irreducible complex characters of $G$, we consider the {\it character degree graph} of \(G\): this is the (simple, undirected) graph whose vertices are the prime divisors of the numbers in $\cd{G}$, and two distinct vertices $p$,  $q$ are adjacent if and only if $pq$ divides some number in $\cd{G}$. This paper completes the classification, started in \cite{DPSS2} and \cite{DPSS}, of the finite non-solvable groups whose character degree graph has a {\it cut-vertex}, i.e. a vertex whose removal increases the number of connected components of the graph. More specifically, it was proved in \cite{DPSS} that these groups have a unique non-solvable composition factor $S$, and that $S$ is isomorphic to a group belonging to a restricted list of non-abelian simple groups. In \cite{DPSS2} and \cite{DPSS} all isomorphism types for $S$ were treated, except the case \(S\cong\PSL{2^a}\) for some integer \(a\geq 2\); the remaining case is addressed in the present paper.
\end{abstract}

\maketitle
\section{Introduction}

The {\it character degree graph} \(\Delta(G)\) of a finite group $G$ is a very useful tool for studying the arithmetical structure of the set \(\cd G=\{\chi(1):\chi\in\irr G\}\), i.e. the set of the irreducible (complex) character degrees of $G$. As many results in the literature show, there is a profound interaction between the group structure of $G$ and certain graph-theoretical properties (in particular, connectivity properties) of \(\Delta(G)\). 

In the papers \cite{DPSS2} and \cite{DPSS} we considered the problem of classifying the finite non-solvable groups $G$ such that $\Delta(G)$ has a {\it cut-vertex}, which is a vertex whose removal (together with all the edges incident to it) produces a graph having more connected components than the original. Among the various properties of such a group $G$, it is proved in \cite{DPSS} that $G$ has a unique non-solvable composition factor $S$, and that $S$ is isomorphic to one of the simple groups in the following list: the projective special linear group $\PSL{t^a}$ (where $t^a$ is a prime power greater than $3$), the Suzuki group ${\rm Sz}(2^a)$ (where $2^a-1$ is a prime number), ${\rm PSL}_3(4)$, the Mathieu group ${\rm M}_{11}$, and the first Janko group ${\rm J}_1$. The aforementioned papers carry out an analysis (and provide a complete classification) of all the possibilities, except for the case $S\cong\PSL{2^a}$ when \(\Delta(G)\) is connected; the present work addresses the remaining case, thus completing the classification of these groups. We refer the reader to \cite{DPSS2} and \cite{DPSS} for a thorough description of the problem and, in particular, for the full statements of the relevant theorems (see the introductions of \cite{DPSS2} and \cite{DPSS}, and Section~2 of \cite{DPSS}).

The situation that remains to be studied is treated in the following Theorem~1 and Theorem~2, which deal with the cases $2^a>4$ and $2^a=4$, respectively (see \cite[Theorem~A, Case~(f)]{DPSS}, and \cite[Theorem~B]{DPSS}), and which are the main results of this paper. 

In order to clarify the statements we mention that, for $H =  {\rm SL}_2(t^a)$ (where $t^a$ is a prime power),  an $H$-module $V$
over the field \(\mathbb{F}_t\) of order $t$ is called {\it{the natural module for $H$}} if $V$ is isomorphic to the standard module for ${\rm SL}_2(t^a)$, or any of its Galois conjugates, seen as an $\mathbb{F}_t[H]$-module.
We will freely use this terminology also referred to the conjugation action of a group on a suitable elementary abelian normal subgroup. 
For our purposes, it is important to recall that the standard module for ${\rm SL}_2(t^a)$ is self-dual. 

Also, given a finite group $G$, we denote by $R = R(G)$ the {\it{solvable radical}} (i.e. the largest solvable normal subgroup), and
by $K = K(G)$ the {\it{solvable residual}} (i.e. the smallest normal subgroup with a solvable factor group) of $G$. Equivalently, $K(G)$ is the last term of the derived series of $G$.

\begin{Thm1}
  Let $R$ and $K$ be, respectively, the solvable radical and the solvable residual of the finite group $G$ and assume that
  $G$ has a composition factor $S \cong \SL{2^a}$, with $a \geq 3$.
  Then, $\Delta(G)$ is a connected graph and has a cut-vertex $p$ if and only if $G/R$ is an almost simple group with socle
  isomorphic to $S$,
  $\V G = \pi(G/R) \cup \{ p\}$ and one of the following holds.
  
    \begin{enumeratei}
 \item \(K\cong S\) is a minimal normal subgroup of \(G\); also, {\it either} $p=2$ and $\V{G/K}\cup\pi(G/KR)=\{2\}$, {\it or} $p\neq 2$, $\V{G/K}=\{p\}$, and $G/KR$ has odd order.
\item \(K\) contains a minimal normal subgroup \(L\) of \(G\) such that \(K/L\cong S\) and \(L\) is the natural module for \(K/L\); also, $p\neq 2$, $\V{G/K}=\{p\}$, $G/KR$ has odd order and, for a Sylow $2$-subgroup $T$ of $G$, we have $T'=(T\cap K)'$.
\end{enumeratei}
  In all cases, $p$ is a complete vertex and the unique cut-vertex of $\Delta(G)$.
\end{Thm1}

\begin{Thm2}
  Let $R$ and $K$ be, respectively, the solvable radical and the solvable residual of the finite group $G$ and assume that
  $G$ has a composition factor $S \cong \SL{4}$. 
  Then, $\Delta(G)$ is a connected graph and has a cut-vertex $p$ if and only if $G/R$ is an almost simple group with socle isomorphic to $S$,
  $\V G = \{2,3,5\} \cup \{ p\}$ and one of the following holds.
  
  \begin{enumeratei}
  \item $K$ is isomorphic either to $\SL 4$ or to $\SL 5$, and $\V {G/K} =\{p\}$; if $p=5$, then $K\cong\SL 4$ and $G=K\times R$.
  \item  $K$ contains a minimal normal subgroup \(L\) of $G$ with $|L| = 2^4$. Moreover,  $G = KR$ and
\begin{enumeratei}
 \item[{\rm{(i)}}]either $L$  is  the natural module for \(K/L\), $p \neq 2$ and $\V{G/K}=\{p\}$, 
 \item[{\rm{(ii)}}] or $L$   is isomorphic to  the restriction  to   \(K/L\), embedded  as  $\Omega_4^-(2)$ into
   $\rm{SL}_4(2)$,  of the standard module of $\rm{SL}_4(2)$. Moreover $p = 5$, $G=K\times R_0$ where $R_0=\cent G K$, and  $\V{R_0}=\V{G/K} \subseteq  \{ 5\}$.
\end{enumeratei}
\item $K$ contains a minimal normal subgroup \(L\) of \(G\) such that $K/L$ is isomorphic to $\SL{5}$, and 
\begin{enumeratei}
 \item[{\rm{(i)}}]  either $L $ is the natural module for \(K/L\), $p\not =5$ and $\V {G/K}=\{p\}$,
  \item[{\rm{(ii)}}] or \(L\) is isomorphic to the restriction to  \(K/L\),  embedded in  \({\rm{SL}}_4(3)\), of the standard module of $\rm{SL}_4(3)$,  $p=2$ and $\V {G/K}\sbs \{2\}$.
\end{enumeratei}
\end{enumeratei} 
In all cases, \(p\) is a complete vertex and the unique cut-vertex of \(\Delta(G)\).

\end{Thm2}

To conclude this introduction, we display in Table~1 the graphs related to the groups as in Theorem~1 and Theorem~2, so, all the possible connected graphs having a cut-vertex $p$, of the form \(\Delta(G)\) where $G$ is a finite group with a composition factor isomorphic to $\SL{2^a}$, $a\geq 2$. The first row of the table shows the graphs arising from Theorem~1, whereas the second row shows the graphs arising from Theorem~2 {\it in the case when $p$ is larger than $5$}. As regards the remaining graphs coming from Theorem~2, they are displayed in the third row of the table, and they are all the paths of length \(2\) with vertex set \(\{2,3,5\}\). Each of them actually occur for groups as in Theorem~2(a) (it is enough to consider the direct product \(\SL 4\times R\) where \(R\) is a non-abelian \(q\)-group, for \(q\in\{2,3,5\}\)). Also, case (b)(ii) is associated to the path \(2-5-3\), and case (c)(ii) to the path \(3-2-5\). 

All the groups considered in the following discussion will be tacitly assumed to be finite groups.

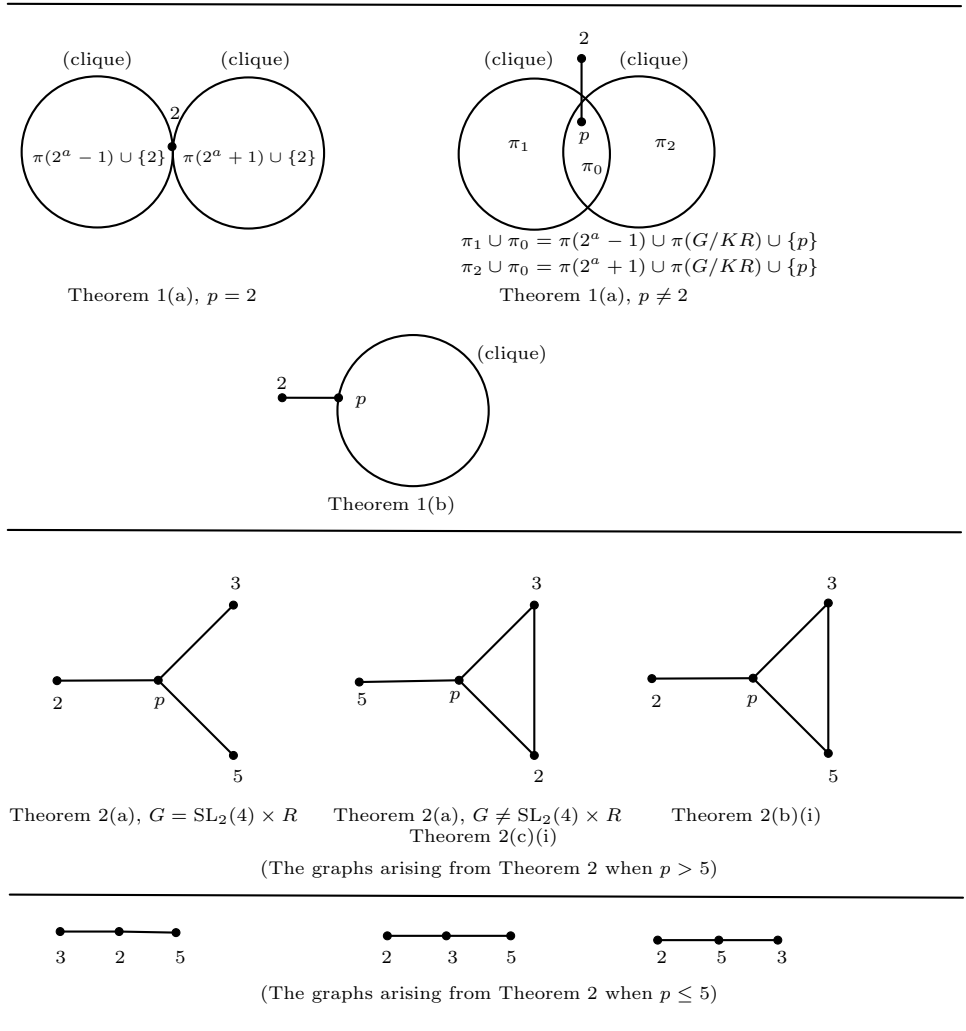
\begin{figure}[h] 
\begin{minipage}{0.9\textwidth}

\begin{tikzpicture}[line cap=round,line join=round,>=triangle 45,x=1.0cm,y=1.0cm]
\clip(-0.8,-3) rectangle (12.845845363984832,11.688625555132022);
\draw (5.6154821265091276,11.5) node[anchor=north west] {\scriptsize TABLE 1};
\draw [line width=0.8pt] (0.6545454545454547,1.994545454545455)-- (2.,2.);
\draw [line width=0.8pt] (3.,3.)-- (2.,2.);
\draw [line width=0.8pt] (2.,2.)-- (3.,1.);
\draw (0.4721015492819798,1.9) node[anchor=north west] {\scriptsize $2$};
\draw (1.819878975217169,1.9) node[anchor=north west] { \scriptsize $p$};
\draw (2.837881073529918,3.5) node[anchor=north west] {\scriptsize $3$};
\draw (2.8665571889753476,0.93) node[anchor=north west] {\scriptsize $5$};
\draw [line width=0.8pt] (4.672727272727274,1.9763636363636368)-- (6.,2.);
\draw [line width=0.8pt] (7.,3.)-- (6.,2.);
\draw [line width=0.8pt] (6.,2.)-- (7.,1.);
\draw (4.515433827087548,1.9530843614086506) node[anchor=north west] {\scriptsize$5$};
\draw (5.748506791241018,1.9817604768540802) node[anchor=north west] {\scriptsize$p$};
\draw (6.838199178167342,3.5) node[anchor=north west] {\scriptsize$3$};
\draw (6.866875293612771,0.93) node[anchor=north west] {\scriptsize$2$};
\draw [line width=0.8pt] (8.56286759390931,2.021974094242668)-- (9.908322139363856,2.0274286396972134);
\draw [line width=0.8pt] (10.908322139363856,3.0274286396972134)-- (9.908322139363856,2.0274286396972134);
\draw [line width=0.8pt] (9.908322139363856,2.0274286396972134)-- (10.908322139363856,1.0274286396972132);
\draw (8.42972358538868,1.93) node[anchor=north west] {\scriptsize$2$};
\draw (9.7,1.93) node[anchor=north west] {\scriptsize$p$};
\draw (10.76682699419119,3.5) node[anchor=north west] {\scriptsize$3$};
\draw (10.781165051913904,0.93) node[anchor=north west] {\scriptsize$5$};
\draw (8.6,0.43) node[anchor=north west] {\scriptsize$\text{ Theorem 2(b)(i)}$};
\draw (-0.2,0.43) node[anchor=north west] {\scriptsize$\text{ Theorem 2(a), } G=\SL{4}\times R$};
\draw (4.2,0.43) node[anchor=north west] {\scriptsize$\text{Theorem 2(a), } G\not =\SL{4}\times R$};
\draw (5.12,0.16082714606929635) node[anchor=north west] {\scriptsize$\text{ Theorem~2(c)(i)}$};
\draw [line width=0.8pt] (0.,4.)-- (12.67272727272727,3.969090909090908);
\draw [line width=0.8pt] (0.026534780134183478,-0.85159078922567546)-- (12.699262052861453,-0.8924998801347664);
\draw (3.1389802857069284,-0.27159078922567546) node[anchor=north west] {\scriptsize $\text{ (The graphs arising from Theorem~2 when $p> 5$) }$};
\draw [line width=0.8pt] (7.,3.)-- (7.,1.);
\draw [line width=0.8pt] (10.908322139363856,3.0274286396972134)-- (10.908322139363856,1.0274286396972132);
\draw [line width=0.8pt] (0.,11.)-- (12.672727272727276,10.969090909090909);
\draw [line width=0.8pt] (1.1899742184621827,9.02960989734526) circle (1.0053670359753395cm);
\draw [line width=0.8pt] (3.2,9.015359999363264) circle (1.0053670359753422cm);
\draw (0.2,9.179) node[anchor=north west] {{\tiny $\pi(2^a-1)\cup\{2\}$}};
\draw (2.2,9.179) node[anchor=north west] {{\tiny $\pi(2^a+1)\cup\{2\}$}};
\draw (2.03,9.75) node[anchor=north west] {\scriptsize$2$};
\draw (2.7,10.5) node[anchor=north west] {\scriptsize(clique)};
\draw (0.6,10.5) node[anchor=north west] {\scriptsize(clique)};
\draw (0.67100233710496943,7.358532122872143) node[anchor=north west] {\scriptsize $\text{Theorem 1(a), } p=2$};
\draw [line width=0.8pt] (7.,9.) circle (1.0053670359753404cm);
\draw [line width=0.8pt] (8.39,9.029609897345264) circle (1.0053670359753388cm);
\draw (6.2,10.5) node[anchor=north west] {\scriptsize(clique)};
\draw (8.0,10.5) node[anchor=north west] {\scriptsize(clique)};
\draw (7.497749833412221,9.0) node[anchor=north west] {\scriptsize$\pi_0$};
\draw (7.469073717966791,9.42) node[anchor=north west] {\scriptsize$p$};
\draw [line width=0.8pt] (7.628081865416518,9.43)-- (7.628081865416518,10.272038341433083);
\draw (7.469073717966791,10.75) node[anchor=north west] {\scriptsize$2$};
\draw (6.522761908267617,9.30850797316136) node[anchor=north west] {\scriptsize$\pi_1$};
\draw (8.472737758556825,9.322846030884076) node[anchor=north west] {\scriptsize$\pi_2$};
\draw (5.9,8.08) node[anchor=north west] {\scriptsize $\pi_1\cup\pi_0=\pi(2^a-1)\cup\pi(G/KR)\cup \{p\}$};
\draw (5.9,7.74) node[anchor=north west] {\scriptsize $\pi_2\cup\pi_0=\pi(2^a+1)\cup\pi(G/KR)\cup \{p\}$};
\draw (6.406591830249884,7.358532122872143) node[anchor=north west] {\scriptsize$\text{Theorem 1(a), } p\not=2$};
\draw [line width=0.8pt] (5.388329704963107,5.587731866300052) circle (1.0053670359753395cm);
\draw (6.1069582343088875,6.584277005845542) node[anchor=north west] {\scriptsize(clique)};
\draw (3.45,6.154135274164097) node[anchor=north west] {\scriptsize$2$};
\draw (4.5,5.9) node[anchor=north west] {\scriptsize$p$};
\draw [line width=0.8pt] (3.6484469972163645,5.758867870340715)-- (4.3974175800501625,5.757602516285128);
\draw (4.128306268574248,4.58) node[anchor=north west] {\scriptsize $\text{Theorem 1(b)}$};
\draw [line width=0.8pt] (0.6963509275149233,-1.34327629734679)-- (1.4807242793679631,-1.34327629734679);
\draw [line width=0.8pt] (1.4807242793679631,-1.34327629734679)-- (2.2365749638808925,-1.3575376310168452);
\draw [line width=0.8pt] (5.046057696881782,-1.4003216320270104)-- (5.830431048734821,-1.4003216320270104);
\draw [line width=0.8pt] (5.830431048734821,-1.4003216320270104)-- (6.686111068938137,-1.4003216320270104);
\draw [line width=0.8pt] (8.639913781735709,-1.4573669667072313)-- (9.45280980092886,-1.4573669667072313);
\draw [line width=0.8pt] (9.45280980092886,-1.4573669667072313)-- (10.237183152781899,-1.4573669667072313);
\draw (1.3,-1.48) node[anchor=north west] {\scriptsize$2$};
\draw (4.842561385600847,-1.48) node[anchor=north west] {\scriptsize$2$};
\draw (8.5,-1.48) node[anchor=north west] {\scriptsize$2$};
\draw (0.51,-1.48) node[anchor=north west] {\scriptsize$3$};
\draw (5.719830675795589,-1.48) node[anchor=north west] {\scriptsize$3$};
\draw (10.1,-1.48) node[anchor=north west] {\scriptsize$3$};
\draw (2.1,-1.48) node[anchor=north west] {\scriptsize$5$};
\draw (6.5,-1.48) node[anchor=north west] {\scriptsize$5$};
\draw (9.3,-1.48) node[anchor=north west] {\scriptsize$5$};
\draw [line width=0.8pt] (0.005939202842223601,-2.6814326654051786)-- (12.678666475569502,-2.7123417563142698);
\draw (3.1389802857069284,-1.9336284936240807) node[anchor=north west] {\scriptsize $\text{ (The graphs arising from Theorem~2 when $p\leq 5$) }$};
\begin{scriptsize}
\draw [fill=black] (0.6545454545454547,1.994545454545455) circle (1.5pt);
\draw [fill=black] (2.,2.) circle (1.5pt);
\draw [fill=black] (3.,3.) circle (1.5pt);
\draw [fill=black] (3.,1.) circle (1.5pt);
\draw [fill=black] (4.672727272727274,1.9763636363636368) circle (1.5pt);
\draw [fill=black] (6.,2.) circle (1.5pt);
\draw [fill=black] (7.,3.) circle (1.5pt);
\draw [fill=black] (7.,1.) circle (1.5pt);
\draw [fill=black] (8.56286759390931,2.021974094242668) circle (1.5pt);
\draw [fill=black] (9.908322139363856,2.0274286396972134) circle (1.5pt);
\draw [fill=black] (10.908322139363856,3.0274286396972134) circle (1.5pt);
\draw [fill=black] (10.908322139363856,1.0274286396972132) circle (1.5pt);
\draw [fill=black] (2.18439766419588,9.1) circle (1.5pt);
\draw [fill=black] (7.628081865416518,9.43) circle (1.5pt);
\draw [fill=black] (7.628081865416518,10.272038341433083) circle (1.5pt);
\draw [fill=black] (3.6484469972163645,5.758867870340715) circle (1.5pt);
\draw [fill=black] (4.3974175800501625,5.757602516285128) circle (1.5pt);
\draw [fill=black] (0.6963509275149233,-1.34327629734679) circle (1.5pt);
\draw [fill=black] (1.4807242793679631,-1.34327629734679) circle (1.5pt);
\draw [fill=black] (2.2365749638808925,-1.3575376310168452) circle (1.5pt);
\draw [fill=black] (5.046057696881782,-1.4003216320270104) circle (1.5pt);
\draw [fill=black] (5.830431048734821,-1.4003216320270104) circle (1.5pt);
\draw [fill=black] (6.686111068938137,-1.4003216320270104) circle (1.5pt);
\draw [fill=black] (8.639913781735709,-1.4573669667072313) circle (1.5pt);
\draw [fill=black] (9.45280980092886,-1.4573669667072313) circle (1.5pt);
\draw [fill=black] (10.237183152781899,-1.4573669667072313) circle (1.5pt); 
\end{scriptsize}
\end{tikzpicture}
\end{minipage}

\end{figure}

\section{Preliminaries}

Given a group \(G\), we denote by \(\Delta(G)\) the character degree graph (or {\it degree graph} for short) of \(G\) as defined in the Introduction. Our notation concerning character theory is standard, and we will freely use basic facts and concepts such as Ito-Michler's theorem, Clifford's theory, Gallagher's theorem, character triples and results about extension of characters (see \cite{Is}). 

For a positive integer \(n\), the set of prime divisors of \(n\) will be denoted by \(\pi(n)\), and we simply write \(\pi(G)\) for \(\pi(|G|)\). If \(q\) is a prime power, then the symbol \(\mathbb{F}_q\) will denote the field of order \(q\).

\smallskip
We start by recalling some structural properties of the groups \(\SL{2^a}\). 

\begin{rem}\label{Subgroups}
The group \(\SL{2^a}=\PSL{2^a}\) has order \(2^a(2^a-1)(2^a+1)\), and the proper subgroups of this group are of the following types (\cite[II.8.27]{Hu}):
\begin{itemize}
\item[(i$_+$)] dihedral groups of order $2(2^a+ 1)$ and their subgroups; 
\item[(i$_-$)] dihedral groups of order $2(2^a- 1)$ and their subgroups; 
\item[(ii)] Frobenius groups with elementary abelian kernel of order $2^a$ and cyclic complements of order $2^a -1$, and their subgroups;
\item[(iii)] $A_4$ when $a$ is even or  $A_5$ when $5$ divides $|\SL{2^a}|$; 
\item[(iv)] $\SL{2^b}$, where $b$ is a proper divisor of $a$.  
\end{itemize}  

\end{rem}

When dealing with subgroups of \(\SL{2^a}\), we will refer to the above labels to identify the type of these subgroups. By a subgroup of type (i) we will mean a subgroup that is either of type (i$_-$) or of type (i$_+$).

\begin{lemma}
\label{PSL2}
Let \(G\cong\SL{2^a}\), where \(a\geq 2\). Let \(u\) be a prime divisor of \(2^a-1\), and let \(U\) be a subgroup of \(G\) with \(|U|=u^b\) for a suitable \(b\in\N-\{0\}\). Then \(U\) lies in the normalizer in \(G\) of precisely two Sylow \(2\)-subgroups of \(G\).
\end{lemma}
\begin{proof} See {\cite[Lemma ~2.2]{DPSS2}}.\end{proof}

Next, some properties of the degree graph of simple and almost-simple groups.

\begin{theorem}[\mbox{\cite[Theorem~5.2]{W}}]
  \label{PSL2bis}
Let $S \cong \PSL{t^a}$ or $S \cong \SL{t^a}$, with $t$ prime and $a \geq 1$. 
Let $\rho_{+} = \pi(t^a+1)$ and $\rho_{-} = \pi(t^a-1)$. For a subset $\rho$ of vertices 
of $\Delta(S)$, we denote by $\Delta_{\rho}$ the subgraph of $\Delta = \Delta(S)$ induced
by the subset $\rho$.
Then 
\begin{enumeratei}
\item if $t=2$ and $a \geq 2$, then $\Delta(S)$ has three connected components, $\{t\}$, $\Delta_{\rho_{+}}$ and 
$\Delta_{\rho_{-}}$, and each of them is a complete graph.  
\item if $t > 2$ and $t^a > 5$, then  $\Delta(S)$ has two connected components, 
$\{t\}$ and  $\Delta_{\rho_{+} \cup \rho_{-}}$; moreover, both  $\Delta_{\rho_{+}}$ and $\Delta_{\rho_{-}}$ are
complete graphs, no vertex  in $\rho_{+} - \{ 2 \}$ is adjacent to any vertex  in  
$\rho_{-} - \{ 2\}$ and $2$ is a complete vertex of $\Delta_{\rho_{+} \cup \rho_{-}}$. 
\end{enumeratei}
\end{theorem}

\begin{theorem}\label{MoretoTiep}
  Let $G$ be an almost-simple group with socle $S$, and let $\delta = \pi(G) - \pi(S)$.
  If $\delta \neq \emptyset$, then $S$ is a simple group of Lie type, and every vertex in \(\delta\) is adjacent to every other vertex of $\Delta(G)$ that is not the characteristic of $S$. Moreover, if \(S\cong\SL{2^a}\) and $a \geq 3$, then any prime in \(\pi(G/S)\) is adjacent to every other vertex of \(\Delta(G)\), except possibly to $2$. 
\end{theorem}
\begin{proof}
The first claim is Theorem~3.9 of \cite{DPSS}. As for the second claim, by Theorem~A of \cite{W1} we see that both $(2^a-1)|G/S|$ and  $(2^a+1)|G/S|$ are irreducible character degrees of \(G\).
\end{proof}


\begin{lemma}\label{InfiniteCommutator}
Let \(G\) be a group and let \(R\) be its solvable radical. Assume that \(G/R\) is an almost-simple group with socle isomorphic to $\PSL{t^a}$, for a prime \(t\) with \(t^a> 4\) and \(t^a\neq 9\). Then, denoting by \(K\) the solvable residual of $G$, one of the following conclusions holds.

\begin{enumeratei} 
\item \(K\) is isomorphic to \(\PSL{t^a}\) or to \(\SL{t^a}\); 
\item \(K\) has a non-trivial normal subgroup \(L\) such that \(K/L\) is isomorphic to \(\PSL{t^a}\) or to \(\SL{t^a}\), and every non-principal irreducible character of \(L/L'\) is not invariant in \(K\).
\end{enumeratei}
\end{lemma}

\begin{proof} See {\cite[Lemma ~2.5]{DPSS2}}.\end{proof}




\begin{lemma}\label{extn}
  Let \(G\) be a group, let \(R\) be its solvable radical and $K$ its solvable residual.
  Assume that   \(L\) is a normal subgroup of \(G\), contained in \(K\), such that \(K/L\cong\SL{2^a}\) with $a \geq 2$, and
  \(L\) is isomorphic to the natural module for \(K/L\). Let $T$ be a Sylow $2$-subgroup of $KR$, let
  $T_0 = T \cap R$ and $T_1 = T \cap K$. Then $L \leq \zent{T_0}$.
  Furthermore, every non-principal $T$-invariant character
  $\lambda \in \irr L$ extends to $T_1$ and, assuming that  $T_0/L$ is abelian,  $\lambda$ extends  to $T$ if and only if $T'= T_1'$.  
\end{lemma}
\begin{proof}
  Observe that $L = K \cap R$ is an elementary abelian $2$-group of order $2^{2a}$, $T_0$ is  normalized by $K$
  and $T = T_0T_1$. As $\zent{T_0} \cap L$ is non-trivial and normal in $K$, by the irreducibility of $L$ as a $K$-module
  it follows that $L \leq \zent{T_0}$. 
 
It is well known that \(\norm K{T_1}/L=\norm{K/L}{T_1/L}\) is a subgroup of type (ii) of \(K/L\) whose order is \(2^a\cdot(2^a-1)\) (in fact, $\norm{K/L}{T_1/L}$ can be identified with the subgroup of lower-triangular matrices of \(\SL{2^a}\)); thus we write $\norm K{T_1} = T_1D$, where $D$ is cyclic of order $2^a-1$.  
  By looking at the action of $T_1$ on the natural module $L$, we see that $Z = \zent{T_1} = (T_1)'$ is a normal subgroup, of order $2^a$, of $T_1D$.
  Since $L$ is a self-dual $K$-module, we have $|\cent{\widehat{L}}{T_1}| = |\cent{L}{T_1}| = 2^a = |\widehat{L/Z}|$
  and hence, as certainly the characters of $L/Z$ are $T_1$-invariant, we conclude that the $T_1$-invariant characters of $L$
  are precisely the elements of $\widehat{L/Z}$. They are clearly $T$-invariant and they extend to $T_1$, because $T_1/Z$ is abelian. 

  Let $\lambda \in \irr L$ be a non-principal $T$-invariant character and assume that $\lambda$ has an extension $\tau \in \irr T$. Since $\tau(1) = \lambda(1) = 1$, we have $T' \leq \ker{\tau}$.
  Assuming that $T_0/L$ is abelian, then  $T/L = T_0/L \times T_1/L$ is abelian and $T' \leq L$.
  So, $T' = T' \cap L \leq \ker{\tau_L} = \ker{\lambda}$ and, as $\lambda \neq 1_L$, hence $T' < L$.
  Observing that $T'$ is normalized by $D$ and that $D$ acts irreducibly on $L/Z$, we conclude that $T' \leq Z$
  and, since $Z = T_0' \leq T'$, that $T' = T_0'$.

  Conversely, if $\lambda \in \irr L$ is $T$-invariant, then as observed above $Z \leq \ker{\lambda}$ and, assuming
  $Z = T'$, clearly $\lambda$, seen as a character of $L/Z$, extends to the abelian group $T/Z$.
\end{proof}

\begin{rem}\label{NaturalExtended}
Let \(K\) be a group having a normal subgroup \(L\) with \(K/L\cong\SL{2^a}\) (for \(a\geq 2\)), and such that \(L\) is isomorphic to the natural \(K/L\)-module. Then, as a consequence of the previous lemma, we can see that the graph \(\Delta(K)\) is disconnected with two connected components, whose vertex sets are \(\{2\}\) and \(\pi(K)-\{2\}\) respectively, and which are both complete subgraphs of \(\Delta(K)\). 

In fact, by Theorem~\ref{PSL2bis}, \(\Delta(K/L)\) has three connected components with vertex sets \(\pi(2^a-1)\), \(\pi(2^a+1)\) and \(\{2\}\) respectively, which are all complete subgraphs of \(\Delta(K/L)\). On the other hand, if \(\lambda\) is any non-principal character in \(\irr L\), then \(I_K(\lambda)\) is a Sylow \(2\)-subgroup of \(K\), and Lemma~\ref{extn} guarantees that \(\lambda\) extends to \(I_K(\lambda)\); our claim then easily follows by Clifford's theory. 
\end{rem}

\begin{theorem}\label{0.2}
Let $G$ be a non-solvable group such that $\Delta(G)$ is connected  and it has a  cut-vertex \(p\). Then, denoting by $R$ the solvable radical of $G$, we have that $G/R$ is an almost-simple group such that \(\V G=\pi({G/R})\cup\{p\}\). 
\end{theorem}

\begin{proof} See \cite[Theorem~3.15]{DPSS}.\end{proof}


To conclude this preliminary section, we recall the statements of three crucial results proved in  \cite{DPSS2} and \cite{DPSS}, concerning certain module actions of $\SL{t^a})$. 

Let \(H\) and \(V\) be finite groups, and assume that \(H\) acts by automorphisms on \(V\). Given a prime number  \(q\), we say that {\it{the pair \((H,V)\) satisfies the condition \(\nq\)}} if  $q$ divides $|H: \cent HV|$ and, for every non-trivial \(v\in V\), there exists a Sylow \(q\)-subgroup \(Q\) of \(H\) such that \(Q\trianglelefteq \cent H v\) (see \cite{C}).

If $(H, V)$ satisfies $\nq$ then \(V\) turns out to be an elementary abelian \(r\)-group for a suitable prime \(r\), and \(V\) is in fact an {\it{irreducible}} module for \(H\) over the field $\mathbb{F}_r$ (see~Lemma~4 of~\cite{Z}). 

\begin{lemma}[\mbox{\cite[Lemma~3.10]{DPSS}}] \label{SL2Nq}
  Let $t$, $q$, $r$  be prime numbers, let $H = {\rm SL}_2(t^a)$ (with $t^a \geq 4$)  and  let  $V$ be an $\mathbb{F}_r[H]$-module.
  Then $(H, V)$ satisfies $\nq $ if and only if 
 either $t^a = 5$ and $V$ is the natural module for $H/\cent HV \cong {\rm SL}_2(4)$ or $V$ is faithful and 
  one of the following holds.
  \begin{description}
  \item[(1)] $t = q = r$ and $V$ is the natural $\mathbb{F}_r[H]$-module (so $|V| = t^{2a}$); 
  \item[(2)] $q = r = 3$ and $(t^a, |V|) \in \{(5, 3^4), (13, 3^6)\}$. 
  \end{description}
\end{lemma}

\begin{theorem}[\mbox{\cite[Theorem~3.3]{DPSS2}}] \label{TipoIeIIPieni} Let \(V\) be a non-trivial irreducible module for \(G=\SL{t^a}\) over the field \(\F_q\), where \(t^a\geq 4\) and $q$ is a prime number, \(q\neq t\). For odd primes \(r\in\pi(t^a-1)\) and \(s\in\pi(t^a+1)\) (possibly \(r=q\) or \(s=q\)) let \(R\), \(S\) be respectively a Sylow \(r\)-subgroup and a Sylow \(s\)-subgroup of \(G\), and let \(T\) be a Sylow \(t\)-subgroup of \(G\). Then, considering the sets 
\[V_{I_-}=\{v\in V\,|\, {\textnormal{ there exists }} z\in G {\textnormal{ such that }}R ^z\trianglelefteq\cent G v\},\] \[V_{I_+}=\{v\in V\,|\, {\textnormal{ there exists }} z\in G {\textnormal{ such that }}S ^z\trianglelefteq\cent G v\},\] \[V_{II}=\{v\in V\,|\, {\textnormal{ there exists }} z \in G {\textnormal{ such that }}T^z\trianglelefteq\cent G v\},\] we have that \(V-\{0\}\) strictly contains \(V_{I_-}\cup V_{II}\), \(V_{I_+}\cup V_{II}\), and \(V_{I_-}\cup V_{I_+}\), unless one of the following holds.
\begin{enumeratei}
\item \(G\cong\SL 5\), \(s=3\), \(|V|=3^4\) and \(V\setminus\{0\}=V_{I_+}\),
\item \(G\cong\SL {13}\), \(r=3\), \(|V|=3^6\) and \(V\setminus\{0\}=V_{I_-}\).
\end{enumeratei}
\end{theorem}

\begin{theorem}[\mbox{\cite[Theorem~3.4]{DPSS2}}]
\label{TipoIeII}
Let \(T\) be a Sylow \(t\)-subgroup of \(G\cong\SL{t^a}\) (where \(t^a\geq 4\)) and, for a given odd prime divisor \(r\) of \(t^{2a}-1\), let \(R\) be a Sylow \(r\)-subgroup of \(G\). Assuming that \(V\) is a \(t\)-group such that \(G\) acts by automorphisms (not necessarily faithfully) on \(V\) and \(\cent V G=1\), consider the sets 
\[V_I=\{v\in V\,|\, {\textnormal{ there exists }} x\in G {\textnormal{ such that }}R^x\trianglelefteq\cent G v\},{\textnormal{ {\it and}}}\] \[V_{II}=\{v\in V\,|\, {\textnormal{ there exists }} x \in G {\textnormal{ such that }}T^x\trianglelefteq\cent G v\}.\] Then the following conditions are equivalent.
\begin{enumeratei}
\item \(V_I\) and \(V_{II}\) are both non-empty and \(V-\{1\}=V_I\cup V_{II}\).
\item \(G\cong\SL {4}\), and \(V\) is an irreducible \(G\)-module of dimension \(4\) over \(\mathbb{F}_2\). More precisely, \(V\) is the restriction to $G$, embedded as \(\Omega_4^-(2)\) into ${\rm{SL}}_4(2)$, of the standard module of ${\rm{SL}}_4(2)$.
\end{enumeratei}
\end{theorem}





\section{The structure of the solvable residual}

Let \(G\) be a group having a composition factor isomorphic to \(\SL{2^a}\) (with \(a\geq 2\)), such that \(\Delta(G)\) is connected and has a cut-vertex: as the first step in our analysis, our purpose is to describe the structure of the solvable residual \(K\) of \(G\). In particular we will see that, except for two sporadic cases, either we have \(K\cong\SL{2^a}\), or  \(K\cong\SL{5}\), or \(K\) contains a minimal normal subgroup \(L\) of \(G\) such that either  \(K/L\cong\SL{2^a}\) or  \(K/L\cong\SL{5}\) and \(L\) is the natural module for \(K/L\). 

We collect the main results of this section in the following single statement (which is the counterpart in characteristic $2$ of \cite[Theorem~4.1]{DPSS2}). This will be proved by treating separately the case \(a>2\) and the case \(a=2\), in Theorem~\ref{Char2} and Theorem~\ref{PropA5} respectively.
  
  \begin{theorem}\label{MainAboutK}
Assume that the group $G$ has a composition factor isomorphic to $\SL{2^a}$ with $a\geq 2$, and let $p$ be a prime number. 
Assume also that \(\Delta(G)\) is connected with cut-vertex \(p\). Then, denoting by $K$ the solvable residual of $G$, 
 one of the following conclusions holds.  
\begin{enumeratei} 
\item \(K\) is isomorphic to \(\SL{2^a}\) or to \(\SL{5}\); 
\item \(K\) contains a minimal normal subgroup \(L\) of \(G\) such that \(K/L\) is isomorphic either to \(\SL{2^a}\) or  to  \(\SL{5}\) and \(L\) is the natural module for \(K/L\).
\item \(a=2\), and \(K\) contains a minimal normal subgroup \(L\) of \(G\) such that \(K/L\) is isomorphic to \(\SL{4}\). Moreover, \(L\) is isomorphic to the restriction  to  \(K/L\), embedded  as  $\Omega_4^-(2)$ into $\rm{SL}_4(2)$, of the standard module of $\rm{SL}_4(2)$.
\item \(a=2\), and \(K\) contains a minimal normal subgroup \(L\) of \(G\) such that \(K/L\) is isomorphic to \(\SL{5}\). Moreover, \(L\) is isomorphic to the restriction to  \(K/L\),  embedded in  \({\rm{SL}}_4(3)\), of the standard module of $\rm{SL}_4(3)$.
\end{enumeratei}
\end{theorem}

We will then start by treating the case \(a>2\). Before stating the next theorem we recall that, for $m$ and $n$ integers larger than $1$, a prime divisor $q$ of $m^n-1$ is called a {\it{primitive prime divisor}} if $q$ does not divide $m^b -1$ for all $1 \leq b <n$. In this case, $n$ is the order of $m$ modulo $q$, so $n$ divides $q-1$. 
In view of \cite[Theorem~6.2]{MW}, $m^n - 1$ always has primitive prime divisors except when $n = 2$ and $m= 2^c -1$ for some integer $c$ (i.e. $m$ is a Mersenne number), or when $n=6$ and $m= 2$. 

In the following, for a normal subgroup $N$ of a group $G$, and a character $\theta \in \irr N$, we denote by
$\irr{G| \theta}$ the set of all irreducible characters of $G$ that lie over $\theta$.

\begin{theorem}\label{Char2}
Assume that the group $G$ has a composition factor isomorphic to $\SL{2^a}$ with $a>2$, and let $p$ be a prime number. 
Assume also that \(\Delta(G)\) is connected with cut-vertex \(p\). Then, denoting by $K$ the solvable residual of $G$, 
 one of the following conclusions holds.
\begin{enumeratei} 
\item \(K\) is isomorphic to \(\SL{2^a}\); 
\item \(K\) contains a minimal normal subgroup \(L\) of \(G\) such that \(K/L\) is isomorphic to \(\SL{2^a}\) and \(L\) is the natural module for \(K/L\).
\end{enumeratei}

\end{theorem}

\begin{proof}

 Let $R$ be the solvable radical of $G$. By Theorem \ref{0.2}, we have that $G/R$ is an almost-simple group  with socle isomorphic to $\SL {2^a}$, and  $\V G=\pi({G/R})\cup\{p\}$.
Note that, since \(a>2\), Lemma~\ref{InfiniteCommutator} applies here; so either we get conclusion (a), or \(K\) has a non-trivial normal subgroup \(L\) such that \(K/L\) is isomorphic to \(\SL{2^a}\), and every non-principal irreducible character of \(L/L'\) is not invariant in \(K\). Therefore, we can assume that the latter condition holds.

 Consider then a non-principal \(\xi\) in \(\irr{L/L'}\): as \(I_K(\xi)/L\) is a proper subgroup of \(K/L\cong\SL{2^a}\), its possible structures are described in Remark~\ref{Subgroups}. In particular, if \(2\) is not a divisor of \(|K:I_K(\xi)|\), then \(I_K(\xi)/L\) contains a Sylow \(2\)-subgroup of \(K/L\) as a normal subgroup. Assuming for the moment that this happens for every non-principal \(\xi\in\irr{L/L'}\), Lemma~\ref{SL2Nq} (together with the paragraph preceding it) yields that the dual group \(\widehat{L/L'}\) is the natural module for \(K/L\), and the same holds for $L/L'$ by self-duality; so, in order to get the desired conclusion, we only have to show that \(L'\) is trivial (note that, once this is proved, \(L=\oh 2 K\) is a minimal normal subgroup of \(G\)), and this is what we do next. 

 For a proof by contradiction assume \(L'\neq 1\), and consider a chief factor \(L'/Z\) of \(K\). As observed in Remark~\ref{NaturalExtended}, the graph \(\Delta(K/L')\) has two connected components having vertex sets \(\{2\}\) and $\pi(K/L')-\{2\}$, respectively; since the vertex set of \(\Delta(G)\) is \(\pi(G/R)\cup\{p\}\) and, also in view of Theorem~\ref{MoretoTiep}, \(\pi(G/R)-\{2\}\) is now a clique of \(\Delta(G)\), we see that the cut-vertex \(p\) of \(\Delta(G)\) cannot be \(2\), and
 that $p$ is the unique vertex adjacent to  \(2\) in \(\Delta(G)\). 

Now, let \(\lambda\) be a non-principal irreducible character of \(L'/Z\), and let \(\chi\in\irr{K/Z\, |\, \lambda}\). If \(\psi\) is an irreducible constituent of \(\chi_{L/Z}\) lying over \(\lambda\), then clearly \(\psi(1)\neq 1\), and since \(L'/Z\) is an abelian normal subgroup of \(L/Z\) whose index is a \(2\)-power, we conclude that \(\psi(1)\) (whence \(\chi(1)\)) is a multiple of \(2\). As a consequence, we get \(\pi(|K:I_K(\lambda)|)\subseteq\{2,p\}\). Observe that \(I_K(\psi)\) is a proper subgroup of \(K\), as otherwise (the Schur multiplier of \(K/L\) being trivial) \(\psi\) would extend to \(K\) yielding a contradiction via Gallagher's theorem; of course \(I_K(\lambda)\) is a proper subgroup of \(K\) as well, unless \(L'/Z\) lies in \(\zent{K/Z}\).

We conclude this part of the proof by considering three situations that are exhaustive, and that all lead to a contradiction.

{\bf{(i)}} \(L'/Z\not\subseteq\zent{L/Z}.\) 

\noindent Consider the normal subgroup \(\cent{L'/Z}{L/L'}\) of \(K/Z\); since \(L'/Z\) is a chief factor of \(K\) and it is not centralized by \(L/L'\), we deduce that \(\cent{L'/Z}{L/L'}\) is trivial. Thus we can apply the proposition appearing in the Introduction of \cite{Cu}, which ensures that the second cohomology group \({\rm {H}}^2(K/L',L'/Z)\) is trivial, and therefore \(K/Z\) is a split extension of \(L'/Z\); in particular, every irreducible character of \(L'/Z\) extends to its inertia subgroup in \(K/Z\). Now, let \(\lambda\) be any non-principal character in \(\irr{L'/Z}\): since \(\pi(|K:I_K(\lambda)|)\subseteq\{2,p\}\), Gallagher's theorem implies that \(I_K(\lambda)/L'\) contains a unique Sylow \(q\)-subgroup of \(K/L'\) for every prime \(q\in\pi(2^{2a}-1)-\{p\}\).  But this yields a contradiction via, for example, Proposition~3.13 of \cite{DPSS}; in fact, according to that result, \(K/L'\) should have a cyclic solvable radical (whereas \(L/L'=\oh 2{K/L'}\) is non-cyclic).

{\bf{(ii)}} \(L'/Z\sbs \zent{L/Z}\), but \(L'/Z\not\subseteq \zent{K/Z}.\)


\noindent First, we note that \(L'/Z\) is a \(2\)-group in this case, as otherwise \(L/Z\) would be isomorphic to the direct product \((L'/Z)\times (L/L')\) and it would then be abelian, a clear contradiction. Also, for a non-principal \(\lambda\) in \(\irr{L'/Z}\), we already observed that \(I_K(\lambda)\) is a proper subgroup of \(K\) such that \(\pi(|K:I_K(\lambda)|)\subseteq\{2,p\}\).

We claim that \(I_K(\lambda)/L\) cannot be a subgroup of type (iv) of \(K/L\) unless it is also of type (iii). In fact, assume \(I_K(\lambda)/L\cong\SL{2^b}\) where \(b>2\) and \(a=bc\) for some \(c>1\). If \(c\) is an odd number, then \(2^b+1\) is a divisor of \(2^a+1\) and it is easy to see that \(\pi(|K:I_K(\lambda)|)\) contains at least two odd primes, not our case. On the other hand, if \(c\) is even, then \(2^{2b}-1\) divides \(2^a-1\) and again (recalling~cite[Proposition 3.1]{MW}) we reach a contradiction unless \(c=2\) and \(p=2^a+1\) (note that \(p\) is neither \(3\) nor \(5\)). Now we look at \(I_K(\psi)\), where \(\psi\) lies in \(\irr{L/Z\,|\,\lambda}\) (recall that \(\psi(1)\) is a multiple of \(2\), and that \(I_K(\psi)\) is contained in \(I_K(\lambda)\) because \(L'/Z\) is central in \(L/Z\)): we have \(\pi(|I_K(\lambda):I_K(\psi)|)\subseteq\{2\}\), and therefore \(I_K(\psi)/L\) is either the whole \(I_K(\lambda)/L\) or it is necessarily isomorphic to \(A_5\). In any case we get the adjacency of \(2\) with odd primes different from \(p\), a contradiction.

So, assume that \(I_K(\lambda)/L\) is of type (iii) isomorphic to \(A_4\): then there must be a prime in \(\pi(|K:I_K(\lambda)|)-\{2,3\}\), and this prime is necessarily \(p\). This forces the \(3\)-part of \(|K/L|\) to be \(3\), yielding the contradiction that either \(2^a-1=3\) or \(2^a+1=3\). On the other hand, let \(I_K(\lambda)/L\) be of type (iii) isomorphic to \(A_5\). If \(\pi(|K:I_K(\lambda)|)-\{3,5\}\subseteq\{2\}\), then either the \(3\)-part or the \(5\)-part of \(|K/L|\) is forced to be \(3\) or \(5\) respectively, and we get a contradiction from the fact that one among \(3\) and \(5\) is \(2^a-1\) or \(2^a+1\); if  \(\pi(|K:I_K(\lambda)|)-\{3,5\}\) contains an odd prime (which is \(p\)), then the \(3\)-part and the \(5\)-part of \(|K/L|\) are \(3\) and \(5\) respectively, and we get the same contradiction as before unless \(3\cdot 5=2^a-1\), i.e. \(K/L\cong\SL{2^4}\) and \(p=17\) is the only vertex adjacent to \(2\) in \(\Delta(G)\).
But in the latter case, taking \(\psi\in\irr{L/Z\,|\,\lambda}\), we see that \(I_K(\psi)\) cannot be a proper subgroup of \(I_K(\lambda)\) (otherwise \(|I_K(\lambda):I_K(\psi)|\) would be divisible by \(3\) or \(5\) and we would get the adjacency between one of these primes and \(2\)); thus, recalling that \(I_K(\psi)\subseteq I_K(\lambda)\), we get \(I_K(\psi)/L\cong A_5\). Working with character triples we now get the adjacency between \(2\) and \(3\), again a contradiction. 
Our conclusion so far is that, for every non-principal \(\lambda\in\irr{L'/Z}\), the subgroup \(I_K(\lambda)/L\) of \(K/L\) is either of type (i) or of type (ii).

Next, assume that \(I_K(\lambda)/L\) is a subgroup of type (i$_+$). Then we get \(p=2^a-1\) and, since \(2\) cannot be adjacent in \(\Delta(G)\) to any prime in \(\pi(2^a+1)\), for every non-principal \(\nu\in\irr{L'/Z}\) the subgroup \(I_K(\nu)/L\) must be either of type (i$_+$) containing a unique Hall \(\pi(2^a+1)\)-subgroup of \(K/L\), or of type (ii) containing a unique Sylow \(2\)-subgroup of \(K/L\). Now, the former situation cannot occur for every \(\nu\), by Lemma~\ref{SL2Nq}; on the other hand, if the latter situation occurs for some non-principal \(\nu\in\irr{L'/Z}\), then we reach a contradiction via Theorem~\ref{TipoIeII} (recall that $L'/Z$ is a $2$-group).
 
If \(I_K(\lambda)/L\) is a subgroup of type (i$_-$) then, as above, for every non-principal \(\nu\in\irr{L'/Z}\), the subgroup \(I_K(\nu)/L\) must be either of type (i$_-$) containing a unique Hall \(\pi(2^a-1)\)-subgroup of \(K/L\) or of type (ii). Observe that if, in the latter case, \(|K:I_K(\nu)|\) is divisible by \(2\), then \(I_K(\nu)/L\) must contain a Hall \(\pi(2^a-1)\)-subgroup of \(K/L\); hence, by the structure of the subgroups of type (ii), \(I_K(\nu)/L\) should contain a full Sylow \(2\)-subgroup of \(K/L\) as well, against the fact that \(|K:I_K(\lambda)|\) is even. Therefore \(I_K(\nu)/L\)  actually contains a (unique) Sylow \(2\)-subgroup of \(K/L\) whenever it is a subgroup of type (ii), and now we reach a contradiction as in the  previous paragraph. 

We conclude that, for every non-principal \(\lambda\in\irr{L'/Z}\), the subgroup \(I_K(\lambda)/L\) of \(K/L\) is of type~(ii), and the same argument as in the paragraph above shows that it must contain a full Sylow \(2\)-subgroup of \(K/L\). This yields (via Lemma~\ref{SL2Nq}) that \(L'/Z\) is the natural module for \(K/L\), so that \(I_K(\lambda)/L\) is a Sylow \(2\)-subgroup of \(K/L\) for every non-principal \(\lambda\in\irr{L'/Z}\). Considering \(\psi\in\irr{L/Z}\) lying over such a \(\lambda\), and recalling once again that \(\psi(1)\) is even and \(I_K(\psi)\subseteq I_K(\lambda)\), Clifford's theory yields that the primes in \(\pi(K/L)\) are pairwise adjacent in \(\Delta(G)\) and, also in view of Theorem~\ref{MoretoTiep}, every odd prime divisor of \(|K/L|\) is a complete vertex of \(\Delta(G)\). This is clearly not compatible with the existence of a cut-vertex of \(\Delta(G)\).
 
{\bf{(iii)}} \(L'/Z\subseteq\zent{K/Z}.\)

\noindent As in case (ii), we have that $L'/Z$ is a $2$-group. If $\lambda$ is a non-principal irreducible character of $L'/Z$, then $\lambda$ is fully ramified with respect to the $K/Z$-chief factor $L/L'$ (see Exercise~6.12 of \cite{Is}); therefore, the unique \(\psi\) in \(\irr{L/Z\,|\,\lambda}\) is such that \(I_K(\psi)=I_K(\lambda)=K\). The fact that the Schur multiplier of \(K/L\) is trivial implies that \(\psi\) extends to \(K\), yielding a clear contradiction via Gallagher's theorem.


\smallskip
To conclude the proof, we will show that \(I_K(\xi)/L\) contains a unique Sylow \(2\)-subgroup of \(K/L\) for every non-principal \(\xi\) in \(\irr{L/L'}\). To this end, we  will proceed through a number of steps.

\smallskip
{{\bf{(a)}} For every non-principal \(\xi\in\irr{L/L'}\), the subgroup \(I_K(\xi)/L\) of \(K/L\) cannot be of type (iv), unless it is also of type (iii).}

\noindent For a proof by contradiction, let \(\xi\in\irr{L/L'}\) be such that \(I_K(\xi)/L\cong\SL{2^b}\) for some \(b> 2\) properly dividing \(a\). 
Thus, \(2\) is a divisor of \(|K:I_K(\xi)|\). Since the Schur multiplier of \(I_K(\xi)/L\) is trivial, \(\xi\) extends to \(I_K(\xi)\) and this yields (via Clifford's correspondence and Gallagher's theorem) that \(2\) is adjacent in \(\Delta(G)\) to every prime in \(\pi(K/L)-\{2\}\). Moreover, taking into account Theorem~\ref{MoretoTiep} (which, together with Theorem~\ref{PSL2bis}, will be freely used from now on and should be kept in mind), also each prime in \(\pi(G/R)-\pi(K/L)\) is adjacent to every prime in \(\pi(K/L)-\{2\}\). Finally, \(2^{2a}-1\) has a primitive prime divisor \(q\) because \(a\neq 3\); this prime $q$, which clearly belongs to \(\pi(2^a+1)\), is a divisor of \(|K:I_K(\xi)|\), so every prime in \(\pi(2^b-1)\) is  adjacent to \(q\) in \(\Delta(G)\). As easily seen, this setting is not compatible with the existence of a cut-vertex of \(\Delta(G)\).

\smallskip
{{\bf{(b)}}} For every non-principal \(\xi\in\irr{L/L'}\), the subgroup \(I_K(\xi)/L\) of \(K/L\) cannot be isomorphic to~\(A_5\). 

\noindent Assume the contrary, and take \(\xi\in\irr{L/L'}\) such that \(I_K(\xi)/L\cong A_5\). 
Working with character triples, we observe that \(\irr{K\,|\, \xi}\) contains characters whose degrees are divisible by every prime in \(\pi(|K:I_K(\xi)|)\cup\{3\}\), which contains \(\pi(K/L)-\{5\}\) (note that \(2\) divides $|K:I_K(\xi)|$ because \(2^a>4\)); thus the \(5\)-part of \(|K/L|\) is \(5\), otherwise the primes in \(\pi(K/L)\) would be pairwise adjacent in \(\Delta(G)\), easily contradicting the existence of a cut-vertex of \(\Delta(G)\). Observe also that, since neither \(2^a-1\) nor \(2^a+1\) can be \(5\), there exists an odd prime \(q\) in \(\pi(K/L)-\{5\}\) that is adjacent to \(5\) in \(\Delta(K/L)\); as \(q\) is now a complete vertex in the subgraph of \(\Delta(G)\) induced by \(\pi(G/R)\), we get \(q=p\), and it is readily seen that no other prime divisor of \(|K/L|\) can be adjacent to \(5\) in \(\Delta(G)\). This implies on one hand that \(\xi\) does not have an extension to \(I_K(\xi)\) (otherwise, by Gallagher's theorem, we would get the adjacency between $5$ and $2$ in \(\Delta(G)\)), which in turn yields (via \cite[8.16, 11.22, 11.31]{Is}) that the order of \(L/L'\) is divisible by~\(2\); on the other hand, one among the sets \(\pi(2^a-1)\) and \(\pi(2^a+1)\) is in fact \(\{5,p\}\).


Now, since \(2^{2a}-1\) is divisible by \(5\), we see that \(a\) must be even, so \(2^2-1=3\) divides \(2^a-1\). Assuming for the moment \(\pi(2^a-1)=\{5,p\}\), we then get \(p=3\), and we also note that \(2^a-1\) has a primitive prime divisor (otherwise \(a\) would be \(6\), but \(2^6-1=63\) is not divisible by \(5\)). Certainly \(3\) is not such a divisor, as \(3\) divides \(2^2-1\) and \(a>2\); hence \(5\) is a primitive prime divisor for \(2^a-1\), so we get \(a=4\) and \(K/L\cong\SL {16}\). But in this case, since \(|L/L'|\) is even, we can consider a chief factor \(L/X\) of \(K\) whose order is a \(2\)-power: the dual group of \(V\) of \(L/X\) is then an irreducible module for \(\SL{16}\) over $\F_2$. It is well known (see \cite{BN}, for instance) that such modules all have a dimension belonging to \(\{8,16,32\}\); if \(V\) is the natural module (of dimension \(8\)) for \(K/L\cong\SL{16}\), then the centralizer in \(K/L\) of every non-trivial element of \(V\) is a Sylow \(2\)-subgroup of \(K/L\), yielding the contradiction that \(5\) is adjacent to \(17\) in \(\Delta(G)\). Also, a direct computation with GAP~\cite{GAP} shows that in the modules of dimensions \(16\) and \(32\) there are elements lying in regular orbits for the action of \(K/L\), thus the primes in \(\Delta(K/L)\) would be pairwise adjacent in \(\Delta(G)\). Only one module is left, which has dimension \(8\) and is not the natural module: to handle this, we can see via GAP~\cite{GAP} that in all possible isomorphism types of extensions of \(V\) by \(\SL{16}\) the set of irreducible character degrees is \(\{1, 15, 16, 17, 51, 68, 204, 255, 272, 340\}\), so \(\pi(K/L)\) would again be a set of pairwise adjacent vertices of \(\Delta(G)\).

It remains to consider the case when \(5\) divides \(2^a+1\), hence \(\pi(2^a+1)=\{5,p\}\), and again we choose a chief factor \(L/X\) of \(K\) that is a \(2\)-group. Now, the dual group of \(L/X\) can be viewed as a (non-trivial) irreducible \(K/L\)-module over \(\F_2\), and if \(T/L\) is a Sylow \(2\)-subgroup of \(K/L\), then clearly there exists a non-principal \(\mu\) in \(\irr{L/X}\) which is fixed by \(T/L\) (so, such that \(I_K(\mu)/L\) contains \(T/L\)); as \(I_K(\mu)/L\) is a proper subgroup of \(K/L\), the only possibility for \(I_K(\mu)/L\) is to be of type (ii). Moreover, since \(5\) is only adjacent to \(p\) in \(\Delta(G)\), no prime divisor of \(2^a-1\) lies in \(\pi(|K:I_K(\mu)|)\), thus in fact \(I_K(\mu)/L=\norm{K/L}{T/L}\) has irreducible characters of degree \(2^a-1\).   Now, \(\mu\) does not extend to \(I_K(\mu)\), as otherwise (by Gallagher's theorem and Clifford correspondence) we would get adjacencies in \(\Delta(G)\) between \(5\) and all the primes in \(\pi(2^a-1)\); but then, for \(\psi\in\irr{I_K(\mu)\,|\,\mu}\) and \(\theta\) an irreducible constituent of \(\psi_{T/X}\) lying over \(\mu\), we have that \(\theta(1)\) is a \(2\)-power larger than \(1\) (otherwise \(\psi\) would be an extension of \(\mu\) to \(T\), and \(\mu\) would then extend to the whole \(I_K(\mu)\)). As a consequence, \(2\) divides \(\psi(1)\) and we get the adjacency in \(\Delta(G)\) between \(5\) and \(2\). This is the final contradiction that rules out the case \(I_K(\xi)/L\cong A_5\).

\smallskip
{{\bf{(c)}}} For every non-principal \(\xi\in\irr{L/L'}\), the subgroup \(I_K(\xi)/L\) of \(K/L\) cannot be isomorphic to~\(A_4\). 

\noindent  Assume \(I_K(\xi)/L\cong A_4\) for some \(\xi\in\irr{L/L'}\). Then we see at once that the primes in \(\pi(K/L)-\{3\}\) are pairwise adjacent in \(\Delta(G)\), thus the $3$-part of \(|K/L|\) is \(3\) and we get the same conclusions as in the first paragraph of (b) with \(3\) in place of \(5\): the cut-vertex \(p\) is an odd prime and it is the unique neighbour of \(3\) in \(\Delta(G)\) among the primes in \(\pi(K/L)\), and one among the sets \(\pi(2^a-1)\) and \(\pi(2^a+1)\) is \(\{3,p\}\).


Assuming first \(\pi(2^a-1)=\{3,p\}\), we see that \(a\neq 6\) because the \(3\)-part of \(2^6-1\) is \(3^2\). Hence \(2^a-1\) has a primitive prime divisor, which is necessarily \(p\). Note that \(a\) cannot be a prime number, as otherwise it would be odd and \(K/L\) would not have subgroups isomorphic to \(A_4\); moreover, If \(k\) is a divisor of \(a\) such that \(1<k<a\), then \(2^k-1\) divides \(2^a-1\) and is coprime to \(p\), so \(2^k-1\) must be \(3\) and \(k\) is \(2\). We conclude that \(a\) is \(4\), so \(K/L\cong\SL{16}\), and we reach a contradiction as in the second paragraph of (b). 

As regards the case \(\pi(2^a+1)=\{3,p\}\), the same argument as in the last paragraph of (b) (replacing \(5\) with \(3\)) completes the proof. 

\smallskip
{{\bf{(d)}}} The subgroups \(I_K(\xi)/L\) of \(K/L\), for \(\xi\) non-principal in \(\irr{L/L'}\), cannot be all of type~(ii) and of ever order, unless each of them contains a (unique) Sylow \(2\)-subgroup of \(K/L\).


\noindent Let us assume that all the subgroups \(I_K(\xi)/L\) of \(K/L\) (for \(\xi\) non-principal in \(\irr{L/L'}\)) are of type~(ii) and of even order, but there exists \(\xi_0\in\irr{L/L'}\) such that \(2\) divides \(|K:I_K(\xi_0)|\). In this setting we observe that \(2^a-1\) does not divide \(|I_K(\xi_0)/L|\), because \(I_K(\xi_0)/L\) is a Frobenius group whose kernel is its unique Sylow \(2\)-subgroup \(T_0/L\), and we are assuming \(|T_0/L| =2^f<2^a\). Therefore there exists \(r\in\pi(2^a-1)\cap\pi(|K:I_K(\xi_0)|)\), and Clifford's correspondence yields that \(\{2,r\}\cup\pi(2^a+1)\) is a set of pairwise adjacent vertices of \(\Delta(G)\). It follows that \(r\) is adjacent in \(\Delta(G)\) to every prime in \(\pi(G/R)-\{r\}\), thus \(r\) is the cut-vertex \(p\), and no other prime in \(\pi(2^a-1)\) can have any neighbour in \(\{2\}\cup\pi(2^a+1)\); in particular, no prime in \(\pi(2^a-1)-\{p\}\) shows up as a divisor of \(|K:I_K(\xi)|\) for any \(\xi\in\irr{L/L'}\). Note also that a primitive prime divisor of \(2^a-1\) cannot lie in \(\pi(I_K(\xi_0)/L)\), as otherwise it would divide \(2^f-1\) (and \(f<a\)); so, if \(a\neq 6\), \(p\) is forced to be the unique primitive prime divisor of \(2^a-1\). Observe finally that the \(p'\)-part of \(2^a-1\) is not \(1\), otherwise \(p\) would not be a cut-vertex of \(\Delta(G)\). Thus there exists a prime \(q\in\pi(2^a-1)-\{p\}\) such that, 
for every \(\xi\in\irr{L/L'}\), the subgroup \(I_K(\xi)/L\) contains a Sylow \(q\)-subgroup of \(K/L\).

Furthermore, the character \(\xi_0\) does not extend to \(I_K(\xi_0)\), as otherwise we would get characters in \(\irr{K\,|\,\xi_0}\) whose degree is divisible by \(q\) and every prime in \(\{2\}\cup\pi(2^a+1)\), not our case; so \(|L/L'|\) is even, and there exists a chief factor \(L/X\) of \(K\) whose order is a \(2\)-power. Note that, by the conclusion  in the paragraph above, the subgroups of the kind \(I_K(\xi)/L\) for \(\xi\) non-principal in \(\irr{L/X}\) are not Sylow \(2\)-subgroups of \(K/L\), thus \(L/X\) is not the natural module for \(K/L\); this in turn implies (via Lemma~\ref{SL2Nq}) that, for some non-principal \(\xi\in\irr{L/X}\), \(I_K(\xi)/L\) does not contain a full Sylow \(2\)-subgroup of \(K/L\). In other words, we can assume that \(\xi_0\) is in fact an irreducible character of \(L/L'\) whose kernel has index \(2\) in \(L\).

Assume for the moment that \(a\) is an even number different from \(6\) (say, \(a=2b)\): as \(2^b-1\) is coprime to \(p\), we get that \(2^b-1\) divides the order of (a Frobenius complement of) \(I_K(\xi_0)/L\),  and so \(|T_0/L|-1=2^f-1\) is a multiple of \(2^b-1\). This forces \(f\) to be a multiple of \(b\) and, since \(f<a=2b\), the only possibility is \(f=b\); note that \(T_0/L\) is then a minimal normal subgroup of \(I_K(\xi_0)/L\). 
The fact that \(\xi_0\) does not extend to  its inertia subgroup in \(K\) also implies that \(T_0/\ker{\xi_0}\) is a non-abelian \(2\)-group; thus \(L/\ker{\xi_0}\), which has order \(2\), is in fact the derived subgroup of \(T_0/\ker{\xi_0}\). Moreover, the normal subgroup \(Z/\ker{\xi_0}=\zent{T_0/\ker{\xi_0}}\) of \(I_K(\xi_0)/\ker{\xi_0}\) cannot be larger than \(L/\ker{\xi_0}\), because \(T_0/L\) is a minimal normal subgroup of \(I_K(\xi_0)/L\) and clearly \(Z/L\) is not the whole \(T_0/L\). We deduce that \(T_0/\ker{\xi_0}\) is an extraspecial \(2\)-group, so (\(b\) is even and) an application of \cite[II, Satz~9.23]{Hu} yields the contradiction that \(2^b-1\) divides \(2^{b/2}+1\).

If \(a=6\), then \(p\) can be either \(3\) or \(7\). In the former case, \(7\) divides \(|I_K(\xi_0)/L|\) and so \(T_0/L\) has order \(2^3\); the same argument as above shows that \(T_0/\ker {\xi_0}\) is an extraspecial \(2\)-group, a clear contradiction. On the other hand, if \(p=7\), then \(|I_K(\xi_0)/L|\) should be a multiple of \(9\), but \(9\) is not a divisor of \(2^f-1\) for any \(f<6\), contradicting the fact that \(I_K(\xi_0)/L\) is a Frobenius group with kernel \(T_0/L\).

It remains to treat the case when \(a\) is odd. In this case, we start by fixing a Sylow \(q\)-subgroup \(Q\) of \(K/L\): if a non-principal \(\xi\in\irr{L/X}\) is stabilized both by \(Q\) and by another \(Q_1\in\syl q{K/L}\), then \(Q\) and \(Q_1\) are contained in the same subgroup of type (ii) of \(K/L\), whence in the normalizer of a suitable Sylow \(2\)-subgroup of \(K/L\). By Lemma~\ref{PSL2}, \(Q\) normalizes precisely two Sylow \(2\)-subgroups of \(K/L\), and since these normalizers contain a total number of \(2^a\) Sylow \(q\)-subgroups each, there are at most \(2(2^a-1)\) choices for \(Q_1\). On the other hand, the total number of Sylow \(q\)-subgroups of \(K/L\) is \(2^{a-1}(2^a+1)\), so there certainly exists an element \(h\in K/L\) such that no non-trivial element in the dual group \({\widehat{L/X}}\) of \(L/X\) is centralized by both \(Q\) and \(Q^h\). As a consequence, setting \(|L/X|=2^d\), we get \(|\cent{{\widehat{L/X}}}{Q}|\leq 2^{d/2}\), and then \[2^d-1\leq(2^{d/2}-1)\cdot 2^{a-1}\cdot(2^a+1).\]
It is easily checked that the above inequality yields \(d<4a\) and, since \(a\) is odd, Lemma~3.12 in \cite{PR} (whose hypotheses require \(d\leq 3a\), but whose proof works assuming \(d<4a\) as well) leaves only one possibility for the isomorphism type of the \(K/L\)-module \({\widehat{L/X}}\) over \(\F_2\). First of all, \(a\) is a multiple of \(3\) (say \(a=3c\)) and \(d=8c\); then, denoting by \(R(1)\) the natural module for \(K/L\) over \(\F_{2^a}\) and by \(\omega\) an automorphism of order \(3\) of \(\F_{2^a}\), we have that \({\widehat{L/X}}\) is a ``triality module", which can be described as follows. Start from the \(K/L\)-module \(V=R(1)\otimes R(1)^\omega\otimes R(1)^{\omega^2}\) over \(\F_{2^a}\) (or one of its Galois twists), and observe that the field of values of (the character of) \(V\) is \(\F_{2^c}\); now, restricting the scalars to \(\F_{2^c}\), \(V\) is a homogeneous \(K/L\)-module and we take an irreducible constituent of it. This irreducible constituent remains irreducible if the scalars are restricted further to $\F_2$, and this is the \(\F_{2}[K/L]\)-module we are considering. In order to finish the proof for this case, it will be enough to show that there exist non-trivial elements of \(V\) whose centralizer in \(K/L\) is not a subgroup of type (ii).

Recall that the elements of \(\SL{2^a}\) whose order is a divisor of \(2^a+1\) are conjugate to elements of the form \(x=\left(
\begin{array}{cc}
0 & 1 \\
1 & \lambda
\end{array}
\right)\), where \(\lambda=\mu+\mu^{2^a}\) for $\mu\in \F_{2^{2a}}-\{1\}$ such that \(\mu^{2^a+1}=1\). The action of such an \(x\) on \(V\) is of course given by the Kronecker product \[\left(
\begin{array}{cc}
0 & 1 \\
1 & \lambda
\end{array}
\right)\otimes\left(
\begin{array}{cc}
0 & 1 \\
1 & \lambda^{\omega}
\end{array}
\right)\otimes\left(
\begin{array}{cc}
0 & 1 \\
1 & \lambda^{\omega^2}
\end{array}
\right).\] Now, setting \(\K=\F_{2^{2a}}\) and \(V^{\K}=V\otimes\K\), we have \(\dim_{\K}\cent{V^{\K}}x=\dim_{\F_{2^a}}\cent V x\); moreover, \(V^{\K}=R(1)^{\K}\otimes(R(1)^{\K})^\omega\otimes (R(1)^{\K})^{\omega^2}\), so the action of \(x\) on \(V^{\K}\) is expressed by the same Kronecker product as above. But \(x\) is conjugate to \(\left(
\begin{array}{cc}
\mu & 0 \\
0 & \mu^{-1}
\end{array}
\right)\) in \(\SL{2^{2a}}\), so our aim is in fact to find \(\mu\) such that the matrix \[\left(
\begin{array}{cc}
\mu & 0 \\
0 & \mu^{-1}
\end{array}
\right)\otimes\left(
\begin{array}{cc}
\mu^{2^c} & 0 \\
0 & \mu^{-2^c}
\end{array}
\right)\otimes\left(
\begin{array}{cc}
\mu^{2^{2c}} & 0 \\
0 & \mu^{-2^{2c}}
\end{array}
\right)\] has a non-zero eigenspace for the eigenvalue \(1\). A direct calculation  shows that it is enough to choose \(\mu\) of order \(2^{2c}-2^c+1\).

\smallskip
{{\bf{(e)}}} The subgroups \(I_K(\xi)/L\) of \(K/L\), for \(\xi\) non-principal in \(\irr{L/L'}\), cannot be all of type~(ii) of even order or of type (i\(_-\)) (both types occurring).

\noindent If, assuming the contrary, there exists a non-principal \(\xi_0\in\irr{L/L'}\) such that \(I_K(\xi_0)/L\) is of type~(ii) and of even order, but not containing a full Sylow \(2\)-subgroup of \(K/L\) then, as in (d), we get the following conditions: there exists a prime \(q\in\pi(2^a-1)\) such that, for every non-principal \(\xi\in\irr{L/L'}\), the subgroup \(I_K(\xi)/L\) contains a Sylow \(q\)-subgroup of \(K/L\), and \(\xi_0\) does not extend to \(I_K(\xi_0)\).
Similarly,  no power of $\xi_0$ can extend to its inertia subgroup $I$  in $K$, if $I/L$  it does not contain a full
Sylow $2$-subgroup of $K/L$. Since all  Sylow $q$-subgroups of $K/L$ are cyclic for $q \neq 2$, by~\cite[Theorem~6.26]{Is}
we can actually assume that the order $o(\xi_0)$ in the dual group of $L/L'$ is a power of $2$.
Hence 
\(\xi_0\) is in \(\irr{L/X}\),  for a chief factor \(L/X\) of \(K\) that is a \(2\)-group, and the rest of the argument in (d) goes through. 

On the other hand, if all the inertia subgroups of type (ii) and of even order contain a full Sylow \(2\)-subgroup of \(K/L\), and there is an inertia subgroup of type (i$_-$) whose index in \(K/L\) is divisible by a prime in \(\pi(2^a-1)\) (that is necessarily \(p\)), then again we are in the same situation as in (d): for every \(\xi\in\irr{L/L'}\), the subgroup \(I_K(\xi)/L\) contains a Sylow \(q\)-subgroup of \(K/L\) for a suitable prime \(q\in\pi(2^a-1)-\{p\}\), and \(L/L'\) has even order. Taking a chief factor \(L/X\) of \(K\) that is a \(2\)-group, we are in a position to apply Theorem~\ref{TipoIeII}  together with Lemma~\ref{SL2Nq}, and we get a contradiction. Finally, if all the inertia subgroups of type (ii) and even order contain a Sylow \(2\)-subgroup of \(K/L\), and all those of type (i$_-$) have order divisible by \(2^a-1\), then Theorem~\ref{TipoIeIIPieni} (applied to the action of \(K/L\) on any chief factor \(V=L/X\) of \(K\) of odd order) yields that \(L/L'\) is a \(2\)-group, and now Theorem~\ref{TipoIeII} yields a contradiction.

\smallskip
{{\bf{(f)}}} The subgroups \(I_K(\xi)/L\) of \(K/L\), for \(\xi\) non-principal in \(\irr{L/L'}\), cannot be all of type~(ii) of even order or of type (i\(_+\)) (both types occurring). 

\noindent If, assuming the contrary, there exists a non-principal \(\xi_0\in\irr{L/L'}\) such that \(I_K(\xi_0)/L\) is of type~(ii) but not containing a full Sylow \(2\)-subgroup of \(K/L\), then \(2\) is adjacent in \(\Delta(G)\) to every prime in \(\pi(K/L)-\{2\}\); however, there also exists a prime \(r\in\pi(2^a-1)\) such that \(r\) divides \(|K:I_K(\xi_0)|\), and this \(r\) is adjacent in \(\Delta(G)\) to every other prime in \(\pi(G/R)\); this is incompatible with the existence of a cut-vertex of \(\Delta(G)\). The case when all the inertia subgroups of type~(ii) contain a full Sylow \(2\)-subgroup of \(K/L\), and there is an inertia subgroup of type (i$_+$) whose index in \(K/L\) is divisible by a prime in \(\pi(2^a+1)\) (which must be \(p\)), yields the following situation: every inertia subgroup of type (i$_+$) contains a Sylow \(q\)-subgroup of \(K/L\) for a suitable prime \(q\in\pi(2^a+1)-\{p\}\), and every inertia subgroup of type (ii) is a full normalizer of a Sylow \(2\)-subgroup of \(K/L\). Moreover, \(|L/L'|\)  is even, and again we reach a contradiction via Theorem~\ref{TipoIeII}. Finally, if all the inertia subgroups of type (ii) contain a Sylow \(2\)-subgroup of \(K/L\), and all those of type (i$_+$) have order divisible by \(2^a+1\), then Theorem~\ref{TipoIeIIPieni} (applied to the action of \(K/L\) on any chief factor \(V=L/X\) of \(K\) of odd order) yields that \(L/L'\) is a \(2\)-group, and again Theorem~\ref{TipoIeII} yields a contradiction.

\smallskip
{{\bf{(g)}}} The subgroups \(I_K(\xi)/L\) of \(K/L\), for \(\xi\) non-principal in \(\irr{L/L'}\), cannot be all of type~(ii) with even order, of type (i\(_+\)), or of type (i\(_-\)) (all types occurring). 

\noindent Let us assume the contrary. Then \(2\) is adjacent in \(\Delta(G)\) to all the primes in \(\pi(K/L)-\{2\}\), and any inertia subgroup \(I_K(\xi_0)/L\) of type (ii) is the full normalizer of a Sylow \(2\)-subgroup \(T_0/L\) of \(K/L\), i.e. a Frobenius group of order \(2^a\cdot (2^a-1)\).

Note that \(T_0/L\) is then a minimal normal subgroup of \(I_K(\xi_0)/L\). Moreover, \(\xi_0\) does not extend to \(I_K(\xi_0)\), as otherwise we would get adjacencies between primes in \(\pi(2^a-1)\) and primes in \(\pi(2^a+1)\); hence \(L/L'\) has even order and (as already observed) for  a chief factor \(L/X\) of \(K\) having \(2\)-power order, there exists a character in \(\irr{L/X}\) whose stabilizer in \(K/L\) contains a Sylow \(2\)-subgroup of \(K/L\). In other words, we can assume that \(\xi_0\) lies in \(\irr{L/X}\), so \(|L/\ker{\xi_0}|=2\). Now, \(T_0/\ker{\xi_0}\) is a non-abelian \(2\)-group, thus \(L/\ker{\xi_0}\) is the derived subgroup of \(T_0/\ker{\xi_0}\). Moreover, the normal subgroup \(Z/\ker{\xi_0}=\zent{T_0/\ker{\xi_0}}\) of \(I_K(\xi_0)/\ker{\xi_0}\) cannot be larger than \(L/\ker{\xi_0}\), because \(Z/L\) is not the whole \(T_0/L\). We deduce that \(T_0/\ker{\xi_0}\) is an extraspecial \(2\)-group, so (\(a\) is even and) an application of \cite[II, Satz~9.23]{Hu} yields the contradiction that \(2^a-1\) divides \(2^{a/2}+1\).

\smallskip
{{\bf{(h)}}} The subgroups \(I_K(\xi)/L\) of \(K/L\), for \(\xi\) non-principal in \(\irr{L/L'}\), cannot be all of type (i\(_+\)) or of type (i\(_-\)) (both types occurring). 

\noindent Assuming the contrary, as in the previous case we see that \(2\) is adjacent in \(\Delta(G)\) to all the primes in \(\pi(K/L)-\{2\}\); moreover, all the inertia subgroups are forced to contain either a subgroup of order \(2^a-1\) or a subgroup of order \(2^a+1\). Now, let \(L/X\) be a chief factor of \(K\); by Lemma~\ref{SL2Nq}, both types (i$_+$) and (i$_-$) occur for the inertia subgroups even if we only consider the characters in \(\irr{L/X}\), but then Theorem~\ref{TipoIeIIPieni} yields that \(L/X\) is a \(2\)-group, which is impossible because no non-trivial element of \(\widehat{L/X}\) is centralized by a Sylow \(2\)-subgroup of \(K/L\).

\smallskip
{{\bf{(i)}}} The subgroups \(I_K(\xi)/L\) of \(K/L\), for \(\xi\) non-principal in \(\irr{L/L'}\), cannot be all of type (i\(_+\)).

\noindent Let us assume the contrary, and let \(L/X\) be a chief factor of \(K\). If there exists \(\xi_0\in\irr{L/X}\) such that \(I_K(\xi_0)/L\) does not contain a subgroup of order \(2^a+1\), which means that there exists \(r\in\pi(2^a+1\)) dividing \(|K:I_K(\xi_0)|\), then \(r\) is a complete vertex of \(\Delta(G)\) and it is in fact \(p\). Now \(\pi(2^a+1)-\{p\}\) is forced to contain at least one prime \(q\), and this \(q\) cannot show up in the index of any inertia subgroup \(I_K(\xi)\) in \(K\). In other words, for every non-principal \(\xi\) in \(\irr{L/X}\), the inertia subgroup \(I_K(\xi)/L\) contains a Sylow \(q\)-subgroup of \(K/L\) (as a normal subgroup).

Of course the conclusion of the previous paragraph holds if \(I_K(\xi)/L\) does contain a subgroup of order \(2^a+1\) for every \(\xi\in\irr{L/X}\). Thus, in any case, Lemma~\ref{SL2Nq} applies and we get a contradiction. 

\smallskip
{{\bf{(j)}}} The subgroups \(I_K(\xi)/L\) of \(K/L\), for \(\xi\) non-principal in \(\irr{L/L'}\), cannot be all of type (i\(_-\)).

\noindent This is totally analogous to (i).

\smallskip
As we saw, the only possibility that is left is the desired one: \(I_K(\xi)/L\) contains a unique Sylow \(2\)-subgroup of \(K/L\) for every non-principal \(\xi\) in \(\irr{L/L'}\). The proof is complete.
\end{proof}

\bigskip

Next, we conclude the proof of Theorem~\ref{MainAboutK} addressing the remaining case, i.e. when $a=2$. We start by introducing some notation and a few facts concerning a relevant set of modules.

\smallskip
\begin{itemize}
\item We denote by $V_0$ the natural module for $S = \SL{4}$. We have $|V_0| = 2^4$, $|\cent Sv| = 2^2$ for all non-trivial $v \in V_0$, and the cohomology group  ${\rm H}^2(S,V_0)$ is trivial (whereas ${\rm H}^1(S,V_0) \neq  0$).
  
\item We denote by $V_1$ the restriction to $S = \SL{4}$, embedded as \(\Omega_4^-(2)\) into \({\rm SL}_4(2)\), of the standard module of \({\rm SL}_4(2)\).
We have $|V_1| = 2^4$; moreover, $S$ has two orbits $O_1$ and $O_2$ on $V_1-\{0\}$, and
$\cent Sv \cong S_3$ for $v \in O_1$, while  $\cent Sv \cong A_4$ for $v \in O_2$. As for the relevant cohomology groups, we have ${\rm H}^1(S,V_1) = 0 = {\rm H}^2(S,V_1)$.

\item We denote by $W$ the restriction to $S_1 = \SL{5}$, seen as a subgroup of $\rm{SL}_4(3)$, of the standard module of $\rm{SL}_4(3)$. We have $|W| = 3^4$ and $|\cent {S_1}v| = 3$ for all non-trivial $v \in W$; moreover, ${\rm H}^2(S_1,W) = 0$ .

\item We denote by $U$ the natural module for $S_1 = \SL{5}$. We have $|U| = 5^2$ and $|\cent {S_1}v| = 5$ for all non-trivial $v \in U$; moreover, ${\rm H}^2(S_1,U) = 0$. 
\end{itemize}

\medskip
Note that all the above modules are self-dual: this follows from~\cite[Lemma 3.10]{PR} for $V_0$,  $V_1$ and $U$,  and for $W$ by observing that
$\rm{GL}_4(3)$ has  a unique conjugacy class of subgroups isomorphic to $\SL{5}$.

Finally,
let $B$ be an abelian group and $A$ a group acting on $B$ via automorphisms: we will denote by $\Delta_{orb}(B)$ the graph
whose vertex set is the set of the prime divisors of the set of  orbit sizes $\{ |A:\cent Ab|: b \in B\}$ of the action of $A$ on~$B$, and such that two (distinct) vertices $p$ and $q$ are adjacent if and only if there exists a $b \in B$ such that
the product $pq$ divides $|A:\cent Ab|$.
\begin{lemma}\label{LA5}
  Let $q$ be a prime number and $V$ an elementary abelian $q$-group.
  \begin{enumeratei}
\item If $V$ is a non-trivial irreducible $\SL{4}$-module and 
   the graph $\Delta_{orb}(V)$ is not a clique with vertex set $\{2,3,5\}$, then $q=2$ and  $V$ is isomorphic either to $V_0$ or to $V_1$;
\item If $V$ is a faithful irreducible $\SL{5}$-module and 
 the graph $\Delta_{orb}(V)$ is not a clique with vertex set $\{2,3,5\}$, then either $q=3$ and  $V$ is isomorphic to $W$ or  $q=5$ and $V$ is isomorphic to $U$.   
 \end{enumeratei}
\end{lemma}

\begin{proof}
  By Theorem 2.3 of~\cite{KP}, both  $\SL{4}$ and $\SL{5}$ always have regular orbits on a faithful module of characteristic  $p \geq 7$.
  The remaining cases, of characteristic $p \in \{2,3,5\}$, can be settled by direct computation using GAP \cite{GAP}. 
\end{proof}

\begin{theorem}\label{PropA5}

Assume that the group $G$ has a composition factor isomorphic to $\SL{4}\cong\PSL 5$, and let $p$ be a prime number. Assume also that \(\Delta(G)\) is connected and that it has a cut-vertex \(p\). Then, denoting by $K$ the solvable residual of $G$, one of the following conclusions holds.
\begin{enumeratei} 
\item \(K\) is isomorphic to \(\SL{4}\) or to \(\SL{5}\).
\item \(K\) contains a minimal normal subgroup \(L\) of \(G\) such that \(K/L\) is isomorphic either to \(\SL{4}\) or to \(\SL{5}\) and \(L\) is the natural module for \(K/L\).
\item \(K\) contains a minimal normal subgroup \(L\) of \(G\) such that \(K/L\) is isomorphic to \(\SL{4}\). Moreover, \(L\) is isomorphic to the restriction to \(K/L\), embedded as \(\Omega_4^-(2)\) into \({\rm SL}_4(2)\), of the standard module of \({\rm SL}_4(2)\).
\item \(K\) contains a minimal normal subgroup \(L\) of \(G\) such that \(K/L\) is isomorphic to \(\SL{5}\). Moreover, \(L\) is isomorphic to the restriction to \(K/L\), embedded in \({\rm SL}_4(3)\), of the standard module of \({\rm SL}_4(3)\).
\end{enumeratei}
\end{theorem}

\begin{proof}
  By Lemma~\ref{InfiniteCommutator}  (applied with \(t^a=5\)) either (a) holds, or \(K\) has a non-trivial normal subgroup \(L\) such that \(K/L\) is isomorphic to \(\SL{4}\) or to \(\SL{5}\) and every non-principal irreducible character of \(L/L'\) is not invariant in \(K\). In the latter case, consider a chief factor $L/X$ of \(K\) and set \(V\) to be its dual group; then, taking into account that \(\V G=\V {K}\cup\{p\}\), the hypothesis of \(p\) being a cut-vertex for \(\Delta(G)\) implies that
 the subgraph of $\Delta(G)$ induced by the set of vertices $\{2,3,5\}$ is not a clique. Moreover, \(V\) is a non-trivial irreducible module for \(K/L\), and Clifford's theory yields  that \(\Delta_{orb}(V)\) is not a clique as well. Therefore Lemma~\ref{LA5} applies, and the \(K/L\)-module \(V\) is isomorphic to \(V_0\) or to \(V_1\) if \(K/L\cong\SL 4\) whereas it is isomorphic to \(W\) or to \(U\) if \(K/L\cong\SL 5\). Note that $L/X = \fit {K/X}$ is a chief factor of \(G\) as well, and our proof is complete if \(X=1\).

Working by contradiction, we assume \(X\neq 1\) and we consider a chief factor \(X/Y\) of \(K\): in this situation, we first show that \(X/Y\) is the unique minimal normal subgroup of \(K/Y\). In fact, let \(M/Y\) be another minimal normal subgroup of \(K/Y\). Setting \(N/L=\zent{K/L}\) (and observing that \(N\) is contained in the solvable radical \(R\) of \(G\)), we have that \(K/N\) is the unique non-solvable chief factor of \(K\); so, if \(M/Y\) is non-solvable, then we get \(M/Y\cong K/N\) and hence \(K/Y=M/Y\times N/Y\), contradicting the fact that \(K\) is perfect. Therefore \(M/Y\) is abelian, so the normal subgroup \(MX/X\) of \(K/X\) lies in \(\fit{K/X}=L/X\), and we conclude that \(M/Y\) is contained in \(L/Y\). As a consequence, the \(K/L\)-module \(M/Y\) is isomorphic to \(L/X\), i.e. to one of the \(K/L\)-modules \(V_0\), \(V_1\), \(W\) and \(U\). Now \(L/Y\cong M/Y\times X/Y\) can be regarded as a \(K/L\)-module which is the direct sum of two modules in \(\{V_0,V_1\}\) or two modules in \(\{W,U\}\) (depending on whether \(K/L\cong\SL{4}\) or \(K/L\cong\SL{5}\), respectively); but it is easy to see that \(K/L\) has regular orbits on (the duals of) such modules, and this leads via Clifford's theory to the contradiction that \(\{2,3,5\}\) is a clique of \(\Delta(G)\).
  
Next, suppose that $L/Y$ is nilpotent. Since \(K/Y\) has a unique minimal normal subgroup, clearly \(L/Y\) must be a group of prime-power order and, since \(|L/X|\) is a \(q\)-power for \(q\in\{2,3,5\}\), the same holds for \(|L/Y|\). Furthermore, we have $X/Y \leq \zent{L/Y}$ and, in particular, \(I_{K}(\lambda)\subseteq I_{K}(\mu)\) for every \(\mu\in\irr{X/Y}\) and \(\lambda\in\irr{L/Y\,|\,\mu}\).

  If $q \neq 2$, then $|N/L|=2$ and we claim that $X/Y$ is a non-trivial $K/L$-module.
 In fact, assuming the contrary, we get $X/Y \subseteq \zent{K/Y}$ and $|X/Y| = q$. Observe that \(\cent{L/Y}{N/L}\) is a normal subgroup of \(K/Y\) which contains \(X/Y\) but is not the whole \(L/Y\), so, as \(L/X\) is a chief factor of \(K\), we have \(\cent{L/Y}{N/L}=X/Y\). Now if $L$ is abelian, then   
 by coprime action we get  $L = X/Y \times [L/Y,N/L]$, contradicting the uniqueness of $X/Y$ as a minimal normal subgroup of $K/Y$.
 On the other hand, if $L$ is non-abelian, then $X/Y = (L/Y)' = \zent{L/Y}$ and $L/Y$ is an extraspecial $q$-group.
 So, every non-linear irreducible character of $L/Y$ is $K$-invariant and, since $K/L$ has cyclic Sylow $q$-subgroups, it extends to $K$. It
 easily follows that  $\{2,3,5\}$ is a clique of $\Delta(G)$, a contradiction. 
 Thus the claim is proved, and Lemma~\ref{LA5} applies: our assumption that $q$ is not \(2\) yields then $X/Y \cong L/X \cong U$,
 or $X/Y \cong L/X \cong W$, as $K/L$-modules.
 By the fact that $\{ 2,3,5\}$ cannot be a clique and by the observation in the last sentence of the previous paragraph, it follows that
  $I_{K/L}(\lambda)$ is a Sylow $q$-subgroup of $K/L$ for every $\lambda \in \irr{L/Y}$, a contradiction by the paragraph preceding Lemma~\ref{SL2Nq}. 

  So we can assume $q=2$ and $L = N$. 
  One can check with GAP \cite{GAP}
  that the perfect groups of order $2^5\cdot |\SL{4}|$ always have irreducible characters whose degrees are
  multiple, respectively, of $6$, $10$ and $15$: it follows that $X/Y$ is not the trivial $K/L$-module. 
  Hence by Clifford's theory, together with the fact that \(I_{K}(\lambda)\subseteq I_{K}(\mu)\) for every \(\mu\in\irr{X/Y}\) and \(\lambda\in\irr{L/Y\,|\,\mu}\), the assumptions of Lemma~\ref{LA5} are satisfied for the action of \(K/L\) on \(X/Y\). As a result, $X/Y$ is isomorphic either to $V_0$ or to $V_1$ as a $K/L$-module and,
in particular, we get $|X/Y| = 2^4 = |L/X|$. But again, a direct check via GAP \cite{GAP}shows that the perfect groups of order \(2^8\cdot|\SL{4}|\) all have irreducible characters whose degrees are multiples of \(6\), \(10\), \(15\), yielding the same contradiction as above.

  Finally, we assume that $L/Y$ is non-nilpotent. Thus we have $X/Y= \fit{L/Y} = \fit{K/Y}$, and $|X/Y|$ is coprime to $|L/X|$.
  Observe that $\frat{K/Y} \leq \fit{K/Y} = X/Y$ and that $\frat{K/Y} \neq X/Y$, because otherwise \(K/Y\) modulo its Frattini subgroup would be isomorphic to \(K/X\) and would have a trivial Fitting subgroup, not our case.
  Since $X/Y$ is a minimal normal subgroup of $K/Y$, we deduce that $\frat{K/Y}$ is trivial and hence $X/Y$ has a complement $K_0/Y$ in $K/Y$; in particular, every $\mu \in \irr {X/Y}$ extends to its inertia subgroup $I_K(\mu)$. 
  Let $Z/Y$ be an irreducible $L/Y$-submodule of $X/Y$ (i.e. a minimal normal subgroup of \(L/Y\) contained in \(X/Y\)). Set \(C/Y=\cent{L/Y}{Z/Y}\): as $L/X$ is an elementary abelian \(q\)-group (where \(q\) is a suitable prime in \(\{2,3,5\}\)), the factor group $L/C$ is a cyclic group
  of order $q$ acting fixed-point freely on $Z/Y$. Writing the completely reducible  $L/Y$-module  $X/Y$ as
  $(Z/Y) \times (Z_1/Y)$ for a suitable \(L/Y\)-module \(Z_1/Y\), we consider the character $\mu = \mu_0 \times 1_{Z_1/Y} \in \irr{X/Y}$, where  $\mu_0$ is a non-principal
  irreducible character of $Z/Y$. We observe that $I_{L/Y}(\mu) = C/Y$ and that every $\chi \in \irr{K/Y|\mu}$ has
  a degree divisible by $q$. We also remark that, setting $L_0/Y = (L/Y) \cap (K_0/Y)$,
 if $L_0/Y \cong L/X$ is isomorphic (as a $K/L$-module) either to $V_0$, $V_1$ or $W$, then $|I_{L_0/Y}(\mu)| =|C/X| =|L_0/Y|/q> |L_0/Y|^{1/2}$. 
We claim that, as a consequence, for every prime divisor $r \neq q$ of $|K/L|$, either $r$ divides $|K:I_K(\mu)|$ or $r$ divides the degree of some irreducible character of $I_{K/X}(\mu)$ that lies over $\mu$. In fact, fixing \(R_0/Y\in\syl r{K_0/Y}\), it is not difficult to see that  there exists another Sylow \(r\)-subgroup \(R_1/Y\) of \(K_0/Y\) with \(\langle R_0L_0/L_0, R_1L_0/L_0\rangle=K_0/L_0\) and, since no non-trivial element of \(L_0/Y\) is centralized by the whole \(K_0/L_0\), the dimension over \(\F_{q}\) of the vector space \(\cent{L_0/Y}{R_0L_0/L_0}\) cannot be larger than a half of \(\dim_{\F_q}(L_0/Y)\). 
Now, if $I_{K_0/Y}(\mu)$ (which is isomorphic to $I_{K/X}(\mu)$) contains a Sylow $r$-subgroup $R_0/Y$ of $K_0/Y$ as a normal subgroup, then $R_0/Y$ centralizes $I_{L_0/Y}(\mu)$ because $I_{L_0/Y}(\mu)$ and $R_0/Y$ are normal subgroups of coprime order of $I_{K_0/Y}(\mu)$, and this is not possible as $|I_{L_0/Y}(\mu)| >|L_0/Y|^{1/2}$. 
By Gallagher's theorem, it  hence follows that  $\{2,3,5\}$ is a clique of $\Delta(G)$, a contradiction.

 It only remains the case $L_0/Y \cong U$ (as $K_0/L_0$-module); but in this case $q=5$ divides $\chi(1)$ for every
 $\chi \in \irr{K\,|\, \mu}$, and the Sylow $2$-subgroups and $3$-subgroups of $K_0/L_0$ act fixed point freely on $L_0/Y$.
 Recalling that $I_{L_0/Y}(\mu)$ is normal in $I_{K_0/Y}(\mu)$ and that $|I_{L_0/Y}(\mu)| = 5$,  we hence see that $6$ divides $[K_0:I_{K_0}(\mu)]$, and 
again $\{2,3,5\}$ is a clique of $\Delta(G)$, a contradiction. 
\end{proof}

\section{Proof of Theorem~1}

We are ready to prove Theorem~1, that was stated in the Introduction and that is stated again here, for the convenience of the reader, as Theorem~\ref{main1}.
%
%
%
%
\begin{lemma}\label{star}
  Let $K$ be a normal subgroup of the group $G$ with $K \cong \SL{2^a}$, $a \geq 2$.
  Let $R$ be the solvable radical of $G$ and assume that $\V G =  \pi(G/R) \cup \{ p \}$ for a suitable prime $p$.
  Then
  \begin{enumeratei}
  \item The primes in $\V R$ (if any) are complete vertices of $\Delta(G)$.
  \item If $a \geq 3$ and $2 \in \pi(G/KR)$, then $2$ is a complete vertex of $\Delta(G)$. 
  \item If $2 \not \in \pi(G/KR) \cup \V R$, then $2$ is adjacent in $\Delta(G)$ to a vertex $q$ if and only if
    $q \in \V{G/K}$.    
  \end{enumeratei}
\end{lemma}
\begin{proof}
  We start by proving claim (a). Let $q\in \V R$; as  $KR = K \times R$,  $q$ is adjacent in $\Delta(G)$ to all vertices $\neq q$ in $\V K = \pi(K) = \pi(KR/R)$.
  For $t \in \pi(G/R)- \pi(KR/R)$, by part (a) of Proposition 2.10 of~\cite{DKP} there exists a character $\theta \in \irr K$
  such that $t$ divides $|G:I_G(\theta)|$. Take $\varphi \in \irr R$ such that $q$ divides $\varphi(1)$ and let
  $\psi = \theta \times \varphi \in \irr{KR}$. Since $I_G(\psi) \leq I_G(\theta)$, $tq$ divides $\chi(1)$ for every
  $\chi \in \irr G$ that lies over $\theta$.
  Finally, if $p \in \V G$ but $p \not\in \pi(G/R)$, then $p \in \V R$; so if $q \neq p$, then $q \in \pi(G/R)$ by the assumption on
  $\V G$,   and hence by what we have just proved $q$ is adjacent to $p$ as well. So, $q$ is a complete vertex of $\Delta(G)$.

 We now move to claim (b). Assuming  $2 \in \pi(G/KR)$ and $a \geq 3$, by Theorem~\ref{MoretoTiep} we get that $2$ is adjacent in $\Delta(G)$ to all primes $\neq 2$
  of $\pi(G/R)$; so to $p$ as well if $p \neq 2$ and $p \in \pi(G/R)$.
  On the other hand, if $p \in \V G - \pi(G/R)$, so $p\neq 2$,  then $p \in \V R$ and $p$ is adjacent to $2$ in $\Delta(G)$ by part (a).
  Hence, $2$ is a complete vertex of $\Delta(G)$.

  Finally, we prove claim (c). Assume that $2 \not\in \pi(G/KR) \cup \V R$. Then every character $\chi \in \irr G$ such that $\chi(1)$ is even
  lies over a character $\psi \in \irr{KR}$ with $\psi(1)$ even. Writing $\psi = \alpha \times \beta$ with
  $\alpha \in \irr K$ and $\beta \in \irr R$, since $2 \not \in \V R$ we deduce that $\alpha$ has even degree, and
  hence $\alpha$ is the Steinberg character of $K$. Thus $\alpha$ extends to $G$ (see for instance \cite{S}) and hence $\chi(1) = \alpha(1) \gamma(1) =
  2^a\gamma(1)$ for a suitable $\gamma \in \irr{G/K}$, concluding the proof. 
\end{proof}
\begin{theorem}\label{main1}
  Let $R$ and $K$ be, respectively, the solvable radical and the solvable residual of the group $G$ and assume that
  $G$ has a composition factor $S \cong \SL{2^a}$, with $a \geq 3$.
  Then, $\Delta(G)$ is a connected graph and it  has a cut-vertex $p$ if and only if $G/R$ is an almost simple group with socle
  isomorphic to $S$,
  $\V G = \pi(G/R) \cup \{ p\}$ and one of the following holds.
  \begin{enumeratei}
  \item $K$ is a minimal normal subgroup of $G$, $K \cong S$ and
    either $p=2$ and $\V{G/K} \cup \pi(G/KR) = \{ 2\}$,
    or $p\neq 2$, $\V{G/K} = \{p\}$ and $G/KR$ has odd order. 
  \item $K$ contains a minimal normal subgroup $L$ of $G$ such that $K/L \cong S$, $L$ is  the natural module
    for $K/L$, $p \neq 2$, $\V{G/K} = \{p\}$, $G/KR$ has odd order and,  for a Sylow $2$-subgroup  $T$ of $G$,
    $T' = (T \cap K)'$. 
  \end{enumeratei}
In all cases, $p$ is  is a complete vertex and the only cut-vertex of $\Delta(G)$.
\end{theorem}
\begin{proof}
  We start by proving the ``only if'' part of the statement, assuming that $\Delta(G)$ is connected and that has a cut-vertex $p$.
  Then, by Theorem~\ref{0.2} $G/R$ is an almost-simple group and $\V G = \pi(G/R) \cup \{p\}$. As a consequence, we have that
  the socle $M/R$ of $G/R$ is isomorphic to $S$.
  Let $L = K \cap R$; since  $KR = M$, we see that  $K/L \cong S$.
  
  We observe that by Theorem~\ref{MoretoTiep} every prime in $\pi(G/KR)$ is adjacent in $\Delta(G)$
  to every other vertex in $\Delta(G)$, except possibly $2$ and $p$. Moreover,  part (a) of Lemma~\ref{star} yields that
  $\V R \sbs \{p\}$. 

  We consider first the situation arising when  $L=1$.  Assuming  $p=2$, then $\V G = \pi(G/R)$ and by the above observation we
  deduce that $G/KR$ is a $2$-group and that $\V R \sbs \{2\}$.
  If $G = KR = K \times R$, then,  as $\Delta(G)$ is connected and  $2$ is a cut-vertex of $\Delta(G)$, 
  it immediately follows that $\V R = \{2\}$. So, in any case, $\V{G/K} \cup \pi(G/KR) = \{ 2\}$.
  Assuming instead $p \neq 2$, then (since no vertex in $\V G -\{p \}$ can be complete in $\Delta(G)$)
  part (b) of Lemma~\ref{star} implies that $|G/KR|$ is odd  and it only  remains to show
  that $\V{G/K} = \{p\}$.
  As $\V R \sbs \{p\}$ and $p \neq 2$,  part (c) of Lemma~\ref{star} yields that $2$ is adjacent in $\Delta(G)$ to all primes in $\V{G/K}$, and to them only. 
  As $\Delta(G)$ is connected, it follows that $\V{G/K}$ is non-empty.
  If $q \in \V{G/K}$ and $q\neq p$, then $q$ divides $|G/KR|$ (because
  $\V{KR/K} = \V R \sbs \{ p \}$) and hence, by Theorem~\ref{MoretoTiep}, $q$ (being adjacent also to $2$) would be
  a complete vertex of $\Delta(G)$, a contradiction. Hence, $\V{G/K} = \{p\}$.

  We assume now $L \neq 1$. Then,  by Theorem~\ref{Char2}, $L$ is a minimal normal subgroup of $G$ and $L$ is
  the natural module for $K/L \cong S$.
  By Remark~\ref{NaturalExtended}, the subgraph of $\Delta(G)$ induced by the vertex set $\V G - \{ 2,p\}$ is a
  complete graph. Hence, the assumptions on $\Delta(G)$ imply that $p \neq 2$ and that $2$ is adjacent only to $p$ in $\Delta(G)$.
  Moreover, recalling that $\Delta(G/L)$ is a subgraph of $\Delta(G)$,  by part (a) and part (b) of Lemma~\ref{star}
  we deduce that $2 \not\in \V{R/L} \cup \pi(G/KR)$ and hence, by part (c) of the same lemma, that $\V{G/K}= \{p\}$. 
  Let now $T$ be a Sylow $2$-subgroup of $G$; as $|G/KR|$ is odd, then $T \leq KR$.
  Setting $T_0 = T \cap R$, we observe that $T_0/L$ is an abelian normal  Sylow $2$-subgroup of $R/L$ because $2 \not\in \V{R/L}$.
  Let $T_1 = T \cap K$  and assume, working by contradiction, that $T' \neq T_1'$.
  Let  $\lambda \in \irr L$ be a non-principal character; by Lemma~\ref{extn} $L \leq \zent{T_0}$, so $\lambda$ is
  $T_0$-invariant and, since $L$ is a self-dual $K/L$-module, $I_{K}(\lambda)/L$ is a Sylow $2$-subgroup of $K/L$.
  Hence, since $T = T_0T_1$,   we can assume (up to conjugation) that $\lambda$ is  $T$-invariant.  
  So, by Lemma~\ref{extn}  $\lambda$ has no extension to $T$. As $I_{K}(\lambda)/L = T_1/L$ and $KR/L = K/L \times R/L$,
  $T/L$ is a normal subgroup of $I_{KR}(\lambda)$ and hence $2$ divides the degree of every irreducible character
  $\psi$ of $I_{KR}(\lambda)$ that lies over $\lambda$. By Clifford correspondence, it follows that $2$ is adjacent
  in $\Delta(G)$ to all primes in $\pi(2^{2a} -1) = \pi(|K:I_K(\lambda)|)$,
  a contradiction. Hence, $T' = (T\cap K)'$.

  \medskip
  We proceed now to prove the ``if'' part of the statement and we assume that $G/R$ is an almost simple group
  and that $\V G = \pi(G/R) \cup \{p\}$ for some prime $p$.

  Suppose first  that (a) holds, so  $K$ is a minimal normal subgroup of $G$ and  $K\cong S$. Hence, $KR = K \times R$. 
  Assume that  $p=2$, so $\V G = \pi(G/R)$,   and that $\V{G/K} \cup \pi(G/KR) = \{ 2\}$.  
  If $KR < G$, then $2$ is a complete vertex of $\Delta(G)$ by part (b) of Lemma~\ref{star}, and if $G = KR$, then the same
  is true because in this case $\V R = \V{G/K} = \{2\}$. For $\chi \in \irr G$ and an irreducible constituent $\psi$ of $\chi_{KR}$,
  the odd parts of $\chi(1)$ and of $\psi(1)$ coincide by~\cite[Corollary 11.29]{Is}, so by part (a) of Theorem~\ref{PSL2bis}
  the graph $\Delta(G)-2$, obtained by deleting the vertex $2$ and all incident edges,  has 
  two complete connected components, with vertex sets $\pi(2^a-1)$ and $\pi(2^a+1)$.
  So, $2$ is a cut-vertex of $\Delta(G)$ and, being a complete vertex of $\Delta(G)$, it is the unique cut-vertex of $\Delta(G)$.  
If  $p\neq 2$, $\V{G/K} = \{p\}$ and $G/KR$ has odd order, then (as $R \cong KR/K \nor G/K$) $2 \not\in \V R$ and
by part (c) of Lemma~\ref{star} the vertex $2$ is adjacent only to $p$ in $\Delta(G)$. Hence, $p$ is a cut-vertex of $\Delta(G)$.
We also observe that $p$ is a complete vertex of $\Delta(G)$: this is a consequence of  Theorem~\ref{MoretoTiep} if $p \in \pi(G/KR)$, while if $p \not\in \pi(G/KR)$ the assumption $\V{G/K} = \{p\}$ implies that $p \in \V R$
and hence the claim follows by part (a) of Lemma~\ref{star}. Thus,  $p$ is the unique cut-vertex of $\Delta(G)$.

We assume now that (b) holds, so $K$ contains a minimal normal subgroup $L$ of $G$ such that $K/L \cong S$ and  $L$ is
the natural module for $K/L$. Moreover, $p \neq 2$, $\V{G/K} = \{p\}$, $G/KR$ has odd order and,
for any Sylow $2$-subgroup  $T$ of $G$,  $T' = (T \cap K)'$.
For a non-principal $\lambda \in \irr L$, the  argument used in the fourth paragraph of this proof shows that
$I = I_G(\lambda)$ contains a Sylow $2$-subgroup $T$ of $G$, and $T/L$ is abelian and  normal in $I/L$.
By Lemma~\ref{extn} $\lambda$ extends to $T$ and hence $\lambda$
extends to $I_G(\lambda)$ by~\cite[Theorem 6.26]{Is}.  So, Gallagher's theorem implies that every irreducible character of $G$ that lies over
$\lambda$ has odd degree. We hence deduce that if $\chi \in \irr G$ has even degree, then $\chi \in \irr{G/L}$.
Then, by part (c) of Lemma~\ref{star},   $2$ is adjacent only to $p$ in $\Delta(G)$. So, by  Remark~\ref{NaturalExtended},
the graph obtained by removing the vertex $p$ from $\Delta(G)$ has
two connected components: the single vertex $2$ and the  complete graph with vertex set $\V{G} -\{2,p\}$.
By the discussion of case (a), we know that  $p$ is a complete vertex of $\Delta(G/L)$, hence of $\Delta(G)$; thus, $p$ is  the only cut-vertex of $\Delta(G)$. 
\end{proof}

\section{Proof of Theorem 2}
The last section of this paper is devoted to the proof of Theorem~2, that we state again (in a slightly different form, for technical reasons) as Theorem~\ref{main2}.
\begin{lemma}\label{star2}
  Let $K$ be a normal subgroup of the group $G$ with $K \cong \SL{4}$ or $K \cong \SL{5}$.
  Let $R$ be the solvable radical of $G$, $N = K \cap R$ and assume that $\V G =  \{2,3,5, p \}$ for a suitable prime $p$.
  Then
  \begin{enumeratei}
  \item The primes in $\V{G/K}$ (if any) are complete vertices of $\Delta(G)$.
    
\item If $N \neq 1$ or $KR \neq G$, then $2$ is adjacent to $3$ in  $\Delta(G)$. 
\item If   $5 \not\in \V{G/K}$,   
  then $5$ is adjacent in $\Delta(G)$ exactly to the primes in $\V{G/K}$.
\end{enumeratei}
\end{lemma}
\begin{proof}
  (a):
  By  part (a) of Lemma~\ref{star}  the primes in  $\V{KR/K} = \V{R/N}$ are complete
  vertices of $\Delta(G)$. Let $q \in \V{G/K} - \V{KR/K}$; then for $\chi \in \irr{G/K}$ such that $q$ divides $\chi(1)$ and
  an irreducible constituent $\theta$ of $\chi_{KR/K}$, $q$ divides  $\chi(1)/\theta(1)$ by Clifford's theorem and
  $\chi(1)/\theta(1)$ divides $|G/KR|$ by \cite[Corollary 11.29]{Is}.
  As $|G/KR| \leq 2$, we have $q = |G/KR| = 2$ and $\chi = \theta^{G/K}$.
  Seeing by inflation $\theta \in \irr{KR/N}$ with $K/N \leq \ker{\theta}$, we write
  $\theta = 1_{K/N} \times \psi$, with $\psi \in \irr{R/N}$ and $I_{G/N}(\psi) = I_{G/N}(\theta) = KR/N$.
  So, for every $\varphi\in \irr{K/N}$, $\varphi \times \psi \in \irr{KR/N}$ and $I_{G/N}(\varphi \times \psi) = KR/N$,
  hence $2$ is adjacent to both $3$ and $5$ in $\Delta(G)$. If $p \not\in \{2,3,5\}$, then (since $|N| \leq 2$ and $p \in \V G$)
  $R/N \cong KR/K$ cannot have a normal abelian Sylow  $p$-subgroup, so $p \in \V{KR/K}$ is adjacent to $2$ in $\Delta(G)$
  and $q= 2$ is a complete vertex of $\Delta(G)$.

  \smallskip
  Part (b) is clear,  as both $\SL 5$ and $\rm{Aut}(\SL 4) \cong S_5$ have an irreducible character of degree $6$.

  \smallskip
  (c): Since $|G/KR| \leq 2$, $K R$ contains every Sylow $5$-subgroup of $G$ and, as $5\not\in \V{R}\sbs \V{G/K}$,  if $\chi \in \irr G$ has
    degree divisible by $5$, then $\chi$ lies (both if $K \cong \SL 4$, as well as if $K \cong \SL 5$)
    over the unique character $\alpha \in \irr K$ such that $5$ divides $\alpha(1)$.
    It is easily seen that $\alpha$ extends to $G$. By Gallagher's theorem, we conclude that
    $5$ is adjacent only to the vertices of $\V{G/K}$ in $\Delta(G)$.
\end{proof}
\begin{lemma}\label{star3}
  Let $R$ and $K$ be, respectively, the solvable radical and the solvable residual of the group $G$, and let
$N = R \cap K$. 
  \begin{enumeratei}
    \item If $2\not\in \V{G/K}$, $G = KR$  and $N$ is the natural module for $K/N \cong \SL{4}$, then $N \leq \ker{\chi}$ for
  every $\chi\in \irr G$ such that $\chi(1)$ is even. 
\item   Let  $L \nor G$, $L \leq N$, be  such that $K/L\cong \SL 5$ and $L$ is the natural module for $K/L$.
  If   $5 \not\in \V{G/K}$,  then $5$ is adjacent in $\Delta(G)$ exactly to the primes in $\V{G/K}$.
\end{enumeratei}

\end{lemma}
\begin{proof}
  (a):
 Assume that $2\not\in \V{G/K}$, $G = KR$ and that $N$ is the natural module for $K/N \cong \SL 4$.
  Let $\lambda \in \irr N$ be a non-principal character and let $I = I_G(\lambda)$, $T$ a Sylow $2$-subgroup of $I$,
  $T_0 = T \cap R$ and $T_1 = T \cap K$. Since, by Lemma~\ref{extn}, $I$ contains a Sylow $2$-subgroup of $R$, we see that $T_0\in\syl 2 R$;
  moreover, as $2 \not\in \V{G/K} = \V{R/N}$, $T_0/N$ is abelian and $T_0 \nor R$.
  For $B/N \in \syl 3{K/N}$, as $N \leq \zent{T_0}$ and $[B/N, T_0/N] =1$ by coprimality we get 
  $T_0 = N \cent{T_0}B = N \times \cent{T_0}B$, because $\cent NB= 1$; in particular, $T_0$ is abelian.
  Write $C = \cent{T_0}B$ and $D = \cent{T_0}K$; so $D \nor C$.
  Since  $I\cap K$ is a Sylow $2$-subgroup of $K$, we have $T \in \syl 2G$. 
  As $T = T_1T_0$, we have $T' = T_1' [T_1, T_0] T_0' = T_1'[T_1, T_0]$. We claim that $[T_1, T_0] \leq T_1'$.
  Observing that $[T_1, T_0] = [T_1, N][T_1, C]$, it is enough to prove that $[T_1, C/D]\leq T_1'$.
  Identifying  $C/D$ with a normal  subgroup of  ${\rm{Out}}(K)$,
  one can check (for instance by GAP~\cite{GAP}, as $K = \text{SmallGroup(960,11357)}$) that 
  $[T_1, C/D] \leq [T_1, \oh 2{{\rm{Out}}(K)}] \leq T_1'$,  so the claim follows.
  Hence,  $T' = T_1'$ and by Lemma~\ref{extn} $\lambda$ extends to $T$. Thus, by~\cite[Theorem 6.26]{Is} $\lambda$
  extends to $I$. As $I/N$ has odd index in $G/N$ and has a
  normal abelian Sylow $2$-subgroup, it follows that every irreducible character of $G$ lying over $\lambda$,
  where $\lambda$ is any non-principal character of $N$, has odd degree.  
  
\smallskip
  (b): 
  We observe that $G$ splits over $L$. In fact, if $X$ is a Sylow $2$-subgroup of $N$ (so, $|X| = |N/L|= 2$), then by the Frattini argument
    $G = L \cent GX$ and, as $X$ acts fixed-point-freely on $L$, $L \cap \cent GX = 1$.

    Let $Q_0 \in \syl 5 R$; since $R/N \cong KR/K \nor G/K$, $\V{R/N} \sbs \V{G/K}$ and $5 \not \in \V{R/N}$, so
    $Q_0N/N$ is abelian and normal in $R/N$.
    As $N/L\nor R/L$ and $|N/L|= 2$, $N/L$ is central in $R/L$ and it follows
    that $Q_0/L \nor G/L$, so $Q_0 \nor G$.
    For a non-principal $\lambda \in \irr L$, $I_K(\lambda) = Q_1 \in \syl 5K$.
    So, as $|G/KR| \leq 2$, $Q = Q_0Q_1 \in \syl 5G$ and $Q \leq I = I_G(\lambda)$.
    Since $G$ splits over $L$, $\lambda$ extends to $I$ and, as $Q/L = Q_1/L \times Q_0/L$ is abelian
    and normal in $I/L$, by Gallagher's theorem and Clifford correspondence it follows that $5$ does not divide $\chi(1)$
    for every $\chi \in \irr G$ that lies over $\lambda$. 
    Thus, $L$ is contained in the kernel of every irreducible character of $G$ with degree divisible by $5$,
    and part (c) of Lemma~\ref{star2} applied to $G/L$  yields that $5$ is adjacent in $\Delta(G)$ exactly to the primes in $\V{G/K}$.
\end{proof}

\begin{theorem}\label{main2}
  Let $R$ and $K$ be, respectively, the solvable radical and the solvable residual of the group $G$ and assume that
  $G$ has a composition factor $S \cong \SL{4}$. Let $N = K \cap R$. 
  Then, $\Delta(G)$ is a connected graph and has a cut-vertex $p$ if and only if $G/R$ is an almost simple group with socle isomorphic to $S$,
  $\V G = \{2,3,5\} \cup \{ p\}$ and one of the following holds.
%
\begin{enumeratei} 
\item $K$ is isomorphic either to $\SL 4$ or to $\SL 5$ and  $\V {G/K} =\{p\}$; if $p=5$, then $K \cong \SL 4$ and  $G= K\times R$.
\item  
  $K/N \cong \SL 4$, $|N| = 2^4$, $G = KR$ and one of the following:
  \begin{description}
 \item[(i)]$N$  is  the natural module for \(K/N\), $p \neq 2$,   $\V{G/K}=\{p\}$. 
 \item[(ii)] $N$  isomorphic to the restriction to $K/L$, embedded as
   \(\Omega_4^-(2)\) into $\rm{SL}_4(2)$, of the standard module of $\rm{SL}_4(2)$.
   Moreover,    $p = 5$, $G = K \times R_0$, where $R_0 = \cent GK$, and $\V{R_0} = \V{G/K} \subseteq  \{ 5\}$;
\end{description}
\item  There exists $1\not= L \leq N$, $L$ normal in $G$, with $K/L \cong \SL{5}$ and one of the following: 
\begin{description}
\item[(i)] $|L| = 5^2$,  $L $  is the natural module for
  \(\SL{5}\),  $p\not =5$  and $\V {G/K}=\{p\}$.
\item[(ii)] $|L| = 3^4$, \(L\)  is  the natural module for \(\SL{5}\) seen as a subgroup of \({\rm{GL}}_4(3)\),
  $p=2$ and $\V {G/K}\sbs \{2\}$.
   \end{description}
  \end{enumeratei} 
In all cases, $p$ is  is a complete vertex and the only cut-vertex of $\Delta(G)$.
\end{theorem}
\begin{proof}
  We start by proving the ``only if'' part of the statement, assuming that $\Delta(G)$ is connected and that it  has a cut-vertex $p$.
  Then, by Theorem~\ref{0.2} $G/R$ is an almost-simple group and $\V G = \pi(G/R) \cup \{p\}$. 
  So,   the socle $M/R$ of $G/R$ is isomorphic to  $\SL 4$,  and $\V G = \{ 2,3,5,p\}$.
  Hence, the subgraph of $\Delta(G)$ induced  by  the set of  vertices  $\{2,3,5\}$ cannot be a clique.
  As $N = K \cap R$ and   $KR = M$, then   $K/N\cong M/R \cong \SL 4$ and $|G/KR|\leq 2$.
  Since no vertex of $\Delta(G)$ different from $p$ can be complete, part (a) of Lemma~\ref{star2} implies that $\V{G/K} \sbs \{p\}$.

  We now  apply Theorem~\ref{PropA5}, considering  the  possible  structure types  for  the
  solvable residual $K$ of $G$. 
  
    \smallskip
    If $K$ is isomorphic either to $\SL 4$ or to $\SL 5$ (i.e. $|N| \leq 2$), then part (c) of Lemma~\ref{star2} implies
    (as $5$ cannot be an isolated vertex of $\Delta(G)$) that $\V{G/K}$ is non-empty, so $\V{G/K} = \{ p\}$. 
    By part (b) of Lemma~\ref{star2} $p \neq 5$  when $K \cong \SL 5$ or $KR \neq G$; so we have case (a).

    \smallskip
    Assume now that $|N| > 2$, and that $N$ is a minimal normal subgroup of $G$.
    Then,  by Theorem~\ref{PropA5} $K/N \cong \SL 4$,
    $|N| = 2^4$ and we have two cases:   

    \smallskip
    (x): $N$ is the natural module for $K/N$: then $3$ and $5$ are adjacent in $\Delta(G)$ (see Remark~\ref{NaturalExtended}), and hence $p \neq 2$,
    as othewise $\Delta(G)$ would be a complete graph.  We show that  $G = KR$: in fact,  if this is not the case, then
    $G/R \cong S_5$ and the Sylow $2$-subgroups of $G/N$ are non-abelian. For a non-principal $\lambda \in \irr L$
    and $I = I_G(\lambda)$, $15$ divides $|G:I|$.  Hence, recalling Theorem~A of~\cite{NT}, independently on the parity of $|G:I|$ 
    there exists  $\chi\in \irr G$, lying above $\lambda$, that has degree $30$, a contradiction. Finally, we observe that
    if $G/K \cong R/N$ is abelian, then $2$ is an isolated vertex of $\Delta(G)$, because by part (a) of Lemma~\ref{star3}
    every $\chi \in \irr G$ of even degree is a character of $G/N = K/N \times R/N$. So, $\V{G/K} = \{p\}$ and  we have case (b)(i).
    
    \smallskip
    (xx): $N$ is the restriction to $K/L$, embedded as
   \(\Omega_4^-(2)\) into $\rm{SL}_4(2)$, of the standard module of $\rm{SL}_4(2)$. Then $\Delta(K)$ is the graph $2-5-3$ and hence necessarily $p = 5$.
        Let $R_0 = \cent GK$ and $C = \cent GN$. So, $N \leq C \nor G$ and   $R_0 \leq C \leq R$, since $K/N$ is the only non-solvable composition factor of $G$,  and it acts non-trivially on $N$. 
        As ${\rm{H}}^2(K/N, N) = 0$, $K$ splits over $N$; let $K_0$ be a complement of $N$ in $K$. Note that $R_0 = \cent CK = \cent C{K_0}$.
        We prove that $C = N \times R_0$. It is enough to show that $C = N R_0$, since $\zent K = 1$.
             As $[K, R] \leq N$,  in particular $[K_0, C] \leq N$ and hence
        $K_0^c \leq K_0N = K$ for every $c \in C$.
        Since ${\rm H}^1(K_0, N) = 0$, all complements of $N$ in $K$ are conjugate in $K$. It follows that there exists an
        element $b \in N$ such that $K_0^c = K_0^b$, so $d = bc^{-1} \in \norm C{K_0}$ and hence $[K_0, d] \leq K_0 \cap C = 1$,
        as $K_0 \cong K/N$ acts faithfully on $N$. Thus,   $d \in R_0$. So,  $C = N R_0 = N \times R_0$.

    The action of $G$ on $N$ gives an embedding $\phi$ of $\o{G} = G/C$ in $\widehat{G} = \rm{GL}_4(2)$.
    One can check (for instance by GAP~\cite{GAP}) that $\norm{\widehat{G}}{\phi(\o K)} \cong S_5$,  and that if
    $\phi(\o G)  \cong S_5$ then $\Delta(G/R_0)$, which is a subgraph of $\Delta(G)$, has a complete subgraph with vertex set $\{2,3,5\}$,
    a contradiction. So, $\phi(\o G) = \phi(\o K)$, and hence $G =  K \times R_0$,  giving  case (b)(ii). 

    \medskip
    As the final case, we  assume  that $G$ has a minimal normal subgroup $L$, such that $L \leq N$ and
    $K/L \cong \SL 5$. We have two possible cases:
    
\smallskip     
(y): $L$ is the natural module for $K/L$.
    Then $\Delta(K)$ is the graph with vertex set \{2,3,5\} where $5$ is an isolated vertex and \(2\), \(3\) are adjacent, so we deduce that $p \neq 5$.
    Moreover, part (b) of Lemma~\ref{star3} yields that  $5$ is adjacent in $\Delta(G)$ only to the primes in $\V{G/K}$. 
 Thus, as $\Delta(G)$ is connected,  $\V{G/K} \neq \emptyset$, so $\V{G/K} = \{p \}$ and we have case (c)(i).

 \smallskip
  (yy): $L$ is the natural module for $K/L$ seen as a subgroup of $\rm{GL}_4(3)$. So, $\Delta(K)$ is the graph $3-2-5$ and consequently $p = 2$ and we have case (c)(ii).

    \medskip
    We now prove the ``if'' part of the statement, going through the various cases.

    \smallskip
    (a): If $G \cong \SL 4 \times R$ with $\V R = \V{G/K} = \{5\}$, then clearly $\Delta(G)$ is the graph $2-5-3$.
If $p \neq 5$, then  $5$ is adjacent only to $p$ in $\Delta(G)$ by part (c) of Lemma~\ref{star2}.
 By part (a) of  Lemma~\ref{star2}, $p$ is a complete vertex, and hence the only cut-vertex, of $\Delta(G)$. 

    \smallskip
    (b): We assume that $G = KR$ and that $N = K \cap R$ is a normal in $G$ of order $2^4$.

    In case (b)(i), since $G/N = K/N \times R/N$ and $\V{R/N} = \V{G/K} = \{p\}$ for some prime $p\neq 2$,
    part~(a) of Lemma~\ref{star3} and part (a) of Theorem~\ref{PSL2bis} yield that 
    the vertex $2$ is adjacent only to $p$ in $\Delta(G)$, so $p$ is a cut-vertex of $\Delta(G)$.
    By part (a) of Lemma~\ref{star2}, $p$ is a complete vertex, and hence the only cut-vertex, of $\Delta(G)$.

    In case (b)(ii), it is clear that $\Delta(G) = \Delta(K)$ is the graph $2-5-3$.

    \smallskip
    (c): We assume that there exists $L \nor G$, $L \leq K$,  such that $K/L \cong \SL 5$. 

    In case (c)(i), by part (b) of Lemma~\ref{star3} the vertex $5$ is adjacent only to $p$ ($p \neq 5$) in $\Delta(G)$
    and, by part (a) of the same lemma, $p$ is a complete vertex of $\Delta(G)$.

    In case (c)(ii), we prove that $\Delta(G) = \Delta(K)$, so $\Delta(G)$ is the graph $3-2-5$. To this end, it is enough to show that
    $3$ and $5$ are non-adjacent in $\Delta(G)$. 
    Since $|G/KR| \leq 2$, $KR$ contains a Sylow $3$-subgroup $Q$ of $G$; moreover, as $\V{R/N} \sbs \V{G/K} \sbs \{2\}$ and
    $|N/L| = 2$,
    it easily follows that, setting $Q_0 = Q \cap R$, $Q_0/L$ is abelian  and normal in $R/L$,  and hence in  $G/L$.
    Let $\lambda \in \irr L$  be a non-principal character and let $I = I_G(\lambda)$.
    An application of the Frattini argument, as in the proof of part (b)
    of Lemma~\ref{star3},  proves that $G$ splits over $L$, so  $\lambda$ extends to $I$.
    By~\cite[Lemma 2.6]{DPSS2}, $L \leq \zent{Q_0}$ and hence, since $I \cap K$ is a Sylow $3$-subgroup of $K$, we can assume
    $Q \leq I$. So, $Q/L$ is an abelian Sylow $3$-subgroup of $G/L$ and it is normal in $I/L$.
    Thus,  by Gallagher's theorem we deduce that
    every $\chi \in \irr G$ that lies over $\lambda$ has degree not divisible by $3$.
    Hence, if $\chi \in \irr G$ and $3$ divides $\chi(1)$, then $L \leq \ker{\chi}$ and $\chi \in \irr{G/L}$. 
    Now, an application of part (c) of Lemma~\ref{star2} yields that $5$ not adjacent to $3$ in $\Delta(G/L)$, and
    hence $3$ and $5$ are not adjacent in $\Delta(G)$. 
    
    So, in every case, $p$ is a cut-vertex of $\Delta(G)$ and, as $p$ is  also a complete vertex of $\Delta(G)$, there are
    no other cut-vertices in $\Delta(G)$. The proof is complete. 
  \end{proof}

\end{document}